\documentclass[11pt]{article}

\usepackage{amsmath}
\usepackage{amsmath}
\usepackage{amsthm}
\usepackage{amsfonts}
\usepackage{amssymb}
\usepackage{color}
\usepackage{graphicx}
\usepackage[colorlinks=true,linkcolor=blue,urlcolor=blue,citecolor=red, hyperfigures=false]{hyperref}
\usepackage{dsfont}
\usepackage{enumerate}

\let\inf\relax \DeclareMathOperator*\inf{\vphantom{p}inf}
\let\max\relax \DeclareMathOperator*\max{\vphantom{p}max}
\let\min\relax \DeclareMathOperator*\min{\vphantom{p}min}

\newcommand{\N}{\mathbb N}

\newcommand{\Q}{\mathbb Q}

\newcommand{\R}{\mathbb R}

\newcommand{\E}{\mathbb E}

\def\P{\mathbb P}
\def\F{\mathbb F}

\newcommand{\be}{\begin{equation}}
\newcommand{\ee}{\end{equation}}
\def\1{{\bf 1}}

\def\Dt0{{\bf D}(t_0)}
\def\u{{\bf u}}
\def\v{{\bf v}}

\def\Int{\displaystyle\int}

\def\to{\rightarrow}

\def\pa{\partial}

\def\t{\tau}

\newcommand{\ind}{{\mathds{1}}}

\definecolor{ProcessBlue}{cmyk}{1,0,0,0.40}
\definecolor{Pinegreen}{cmyk}{1,0,1,0.60}

\newcommand{\ba}{\[\begin{array}{rl}}
\newcommand{\ea}{\end{array}\]}
\newcommand{\bea}{\begin{eqnarray}}
\newcommand{\eea}{\end{eqnarray}}
\newcommand{\beaa}{\begin{eqnarray*}}
\newcommand{\eeaa}{\end{eqnarray*}}
\numberwithin{equation}{section}
\numberwithin{figure}{section}
\theoremstyle{plain}
\newtheorem{Theorem}{Theorem}[section]
\newtheorem{Definition}[Theorem]{Definition}
\newtheorem{Proposition}[Theorem]{Proposition}
\newtheorem{Lemma}[Theorem]{Lemma}

\newtheorem{Remark}[Theorem]{Remark}

\newcommand{\AR}{{\cal A}}
\newcommand{\BR}{{\cal B}}

\newcommand{\FR}{{\cal F}}

\newcommand{\LR}{{\cal L}}

\newcommand{\SR}{{\cal S}}

\newcommand{\UR}{{\cal U}}
\newcommand{\VR}{{\cal V}}

\newcommand{\noi}{\noindent}

\newcommand{\iNt}{\mbox{Int}}
\newcommand{\tq}{\, |\,}

\def\pa{\vskip4truept \noindent}

\begin{document}

\title{ A probabilistic representation for the value of zero-sum differential games with incomplete information on both sides}

\author{Fabien Gensbittel$^{(1)}$ and Catherine Rainer$^{(2)}$\\
$\;$\\
\small {(1)} Toulouse School of Economics, University of Toulouse Capitole\\
\small {(2)} Universit\'{e} de Bretagne Occidentale, 6, avenue Victor-le-Gorgeu, B.P. 809, 29285 Brest cedex, France\\
 \small e-mail : 
Fabien.Gensbittel@ut-capitole.fr, Catherine.Rainer@univ-brest.fr}
  \maketitle

\noindent {Abstract : We prove that for a class of zero-sum differential games with incomplete information on both sides, the value admits a probabilistic representation as the value of a zero-sum stochastic differential game with complete information, where both players control a continuous martingale. A similar representation as a control problem over discontinuous martingales was known for games with incomplete information on one side (see Cardaliaguet-Rainer \cite{cr2}), and our result is a continuous-time analog of the so called splitting-game introduced in Laraki \cite{laraki} and Sorin \cite{sorin} in order to analyze discrete-time models. It was proved by Cardaliaguet \cite{c1,c2} that the value of the games we consider is the unique solution of some Hamilton-Jacobi equation with convexity constraints. Our result provides therefore a new probabilistic representation for solutions of  Hamilton-Jacobi equations with convexity constraints as values of stochastic differential games with unbounded control spaces and unbounded volatility.

}
\vspace{3mm}

\vspace{3mm}

\noindent{\bf A.M.S. classification :} 91A05, 91A23, 60G60 
\vspace{3mm}

\section{Introduction.}

In his seminal paper \cite{c1}, Cardaliaguet  introduced a class of zero-sum differential games where each of the two players has a private partial information on the payoff of the game. He showed in \cite{c2} that the value of these games is  the unique solution of some Hamilton-Jacobi-Isaacs (HJI)-equation with convexity constraints. In case only one player has a private information and there are no dynamics, it was established in Cardaliaguet-Rainer  \cite{cr2} that the game has an interpretation in terms of a control problem with complete information over a set of continuous time martingales: these martingales translate how the informed player manages his private information.\\
In the present work, we investigate a game where both players have private information. We show that, in this case, there is an interpretation of the game in terms of a game with complete information and martingale controls. Our model is therefore a continuous-time analog of the so-called ``splitting game'' studied in Laraki \cite{laraki} and Sorin \cite{sorin} (see also De Meyer \cite{demeyergeb} and Gensbittel \cite{gensbittelcavu} for a similar representation formula for the value of repeated games with incomplete information on one side).
But we profit here to go a step further: in opposition to \cite{cr2}, where the controls are merely discontinuous martingales, we prove that the value of the deterministic game with incomplete information is equal to the value of a stochastic differential game (SDG), i.e. that  an -almost- standard Brownian setting is an efficient framework for the analysis of asymmetric information.\\
This way we provide a probabilistic representation of an HJI-equation with convexity constraints of the type introduced in \cite{c2}. The presence of these convexity constraints is strongly linked with the fact that in the class of zero-sum SDG we consider, the control sets and the  volatility map are unbounded. 
A further contribution is the notion of simple pathwise strategies, which provides a less technical alternative to the definition proposed in Cardaliaguet-Rainer \cite{cr3} 

\pa
The introduction is divided in several subsections describing the framework and the known results concerning the initial deterministic game problem, our results and our definition of strategies together with a discussion on the relationships with the existing literature.

\subsection{Continuous time games with incomplete information}\label{introjeuxdiff}
Differential games with incomplete information were introduced by Cardaliaguet in \cite{c1}, as a continuous-time analog to the model of repeated games with incomplete information studied by Aumann and Maschler  in the sixties (see \cite{AumannMaschler} for a re-edition of their work) and by many authors since then (see e.g. Mertens-Zamir \cite{mertenszamir}, Laraki \cite{laraki}, De Meyer-Rosenberg \cite{demeyerrosenberg}, Cardaliaguet-Laraki-Sorin \cite{cls}, Gensbittel \cite{gensbittelcavu} and  Laraki-Sorin \cite{larakisorin} for more references). Let us briefly describe the game whose value equals the value of the zero-sum SDG studied in the present work. \\
We fix two finite sets of indices $I$ and $J$. At the beginning of the game, a pair $(i,j)  \in I\times J$ is chosen at random according to a product distribution $p\otimes q$. Here, $p$ is a probability on the set $I$ and can be assimilated to an element of the simplex $\Delta(I)$ in $\R^{|I|}$ (resp. $q\in\Delta(J)$ a probability on $J$). Player 1 (the minimizing player) is only informed of $i$ while player 2 (the maximizing player) is only informed of $j$. The sets of controls $K$ and $L$ are compact and metric.
The game has a finite time horizon $T$ and an integral payoff. The payoff function $f_{ij}: [0,T]\times K\times L \rightarrow [0,1]$ depends on the chosen pair $(i,j)$. Player 1 is allowed to choose a family of random control processes $k=(k^i)_{i \in I}$. 
 Similarly, player 2 is allowed to choose a family of random controls $\ell=(\ell^j)_{j \in J}$. The expected payoff is defined as 
\[ \sum_{i,j} p_i q_j \E_{\P_1\otimes\P_2} \left[\int_t^T f_{ij}(s,k^i_s,\ell^j_s)ds\right],\]
where $\P_1$ (resp. $\P_2$) is the probability on an auxiliary probability space which is used as a randomization device for Player 1 (resp. Player 2).
The notion of strategies used in \cite{c1} is that of random strategies with delay (see \cite{c1} or \cite{cr2} for a precise definition), and in order for the value to exist, one has to assume the following Isaacs' condition 
\begin{equation}\label{isaacsdgii}
 H(t,p,q):= \sup_k \inf_\ell \sum_{i,j} p_i q_j f_{i,j}(t,k,\ell)=  \inf_\ell \sup_k \sum_{i,j} p_i q_j f_{i,j}(t,k,\ell).
\end{equation}
In Cardaliaguet \cite{c2}, it is shown that the value of this game is 
 the unique Lipschitz continuous viscosity solution of the following equation :
\begin{equation}
\label{edpintro}
\min\left\{\max\left\{ -\frac{\partial V}{\partial t} (t,p,q)-H(t,p,q);-\lambda_{min}(p,D^2_p V(t,p,q))\right\};-\lambda_{max}(q,D^2_q V(t,p,q))\right\}=0,
\end{equation}
with terminal condition $V(T,p,q)=0$. Here the notation $\lambda_{min}(p,D^2_p V(t,p,q))$ stands for the smallest eigenvalue of the second derivative of $V$ with respect to $p$, relative to the tangent space of $\Delta(I)$ at $p$ (see section \ref{model} for the precise definition). Similarly, $\lambda_{max}(q,D^2_q V(t,p,q))$ stands for the largest eigenvalue of  the second derivative of $V$ with respect to $q$, relative to the tangent space of $\Delta(J)$ at $q$. 
The model studied in \cite{c1} led to several generalizations by Cardaliaguet-Rainer \cite{cr1}, Gr\"un \cite{gruen}, Oliu-Barton \cite{oliu}, Gensbittel \cite{gensbittel} and Cardaliaguet-Rainer-Rosenberg-Vieille \cite{crrv}.
 A particular case is deepened in \cite{cr2}: 
when $J$ is reduced to a singleton (i.e. Player 1 has full information), the equation becomes
\begin{equation}
\label{edpintro2}
\max\left\{ -\frac{\partial V}{\partial t} (t,p)-H(t,p);-\lambda_{min}(p,D^2_p V(t,p))\right\}=0,
\end{equation}
with terminal condition $V(T,p)=0$. 
It is then shown that its solution can be represented as
\begin{equation}\label{control2}
 V(t,p)=\min_{(p_s)_{s \in [t,T]}} \E[ \int_t^T H(s,p_s)ds ],
\end{equation}
where the minimum is over the set of (laws of) c\`{a}dl\`{a}g martingales living in the simplex $\Delta(I)$. \\
The martingales $(p_s)_{s \in [t,T]}$ in the control problem (\ref{control2}) can be interpreted as the information on the index $i$ Player 1 discloses over time, and one may derive an optimal strategy for this informed Player from the optimal martingale in (\ref{control2}).

\subsection{Main contributions of the paper}

Our main result is that the unique solution $V$ of \eqref{edpintro} is the value of a standard  SDG (i.e. with complete information) with unbounded control sets and unbounded volatility, providing therefore an alternative probabilistic representation for the value function of differential games with incomplete information. 
In this game, the control sets  $U,V$ are respectively the sets of square matrices of size $|I|$ and $|J|$. Let $(B_s)=(B^1_s,B^2_s)$ denote  a standard Brownian motion with values in $\R^{|I|}\times \R^{|J|}$ defined on the canonical space endowed with the augmented canonical filtration. For fixed $(t,p,q)\in[0,T]\times\Delta(I)\times \Delta(J)$ and a pair of progressively measurable controls $(u_s,v_s)$, we consider the following controlled stochastic differential equations (SDE):
\begin{align}
\label{sdeXY}
X^{t,p,u}_s=p+\int_t^s\sigma(X^{t,p,u}_r,u_r)dB^1_r,\; s\in[t,T], \\
Y^{t,q,v}_s=q+\int_t^s\tau(Y^{t,q,v}_r,v_r)dB^2_r,\; s\in[t,T],
\end{align}
where for all $(x,u)\in\Delta(I)\times U$, $\sigma(x,u) \in\R^{|I\times I|}$ is the orthogonal projection of $u$ on the tangent space of $\Delta(I)$ at $x$, $T_x(\Delta(I))$. Similarly, for all $(y,v)\in\Delta(J)\times V$, $\tau(y,v)\in\R^{|J\times J|}$  denotes the orthogonal projection of $v$ on the tangent space of $\Delta(J)$ at $y$. Roughly speaking, each player controls the variance of his own martingale, $X$ for player $1$ and $Y$ for player $2$, and the projections ensure that the martingales remain respectively in the simplices $\Delta(I)$ and $\Delta(J)$.
\pa 
The payoff of our game is defined by
\[ J(t,p,q,u,v):=\E_t\left[\Int_t^T H(s,X^{t,p,u}_s,Y^{t,q,v}_s)ds\right],\]
where Player 1 plays $u$ and wants to minimize $J(t,p,q,u,v)$ while Player 2 plays $v$ and wants to maximize it. The players use simple pathwise strategies, where simple means here that controls are piecewise-constant on intervals with rational endpoints, and pathwise that each player reacts to the realization of the control of his opponent and not to the full control (see the discussion in subsection \ref{introstrategies}). This notion of strategies allows to define the lower and upper value functions
\[ V^-(t,p,q):=\sup_{\beta}\inf_{\alpha}J(t,p,q,u^{\alpha,\beta},v^{\alpha,\beta}),\qquad  V^+(t,p,q)=\inf_{\alpha}\sup_{\beta}J(t,p,q,u^{\alpha,\beta},v^{\alpha,\beta}) ,\]
where $\alpha,\beta$ range through the set of simple pathwise strategies and $(u^{\alpha,\beta},v^{\alpha,\beta})$ denotes the unique pair of controls induced by the pair of strategies $(\alpha,\beta)$ (see Lemma \ref{fixedpoint}). With such a definition, the inequality $V^-\leq V^+$ is immediate, and we only have to prove the reverse inequality in order to prove that this game admits a value $V=V^-=V^+$. This is the content of Theorem \ref{maintheorem}. 

\pa As a corollary, if $J$ is reduced to a singleton, we find a new stochastic representation for the solution of (\ref{edpintro2}):
\begin{equation}\label{control1}
 V(t,p)=\inf_{u} \E[ \int_t^T H(s,X^{t,u,p}_s)ds ]. 
\end{equation}

One has to compare this result with the representation \eqref{control2}.
 Although we did not follow this approach, one can see \eqref{control2} as the relaxed version  of \eqref{control1} for some sufficiently weak topology. An important advantage of  \eqref{control2} is the existence of an infimum, which is not true in general for \eqref{control1} since the optimal martingales in \eqref{control2} are typically purely discontinuous (see the examples in \cite{cr2}). However, the formulation \eqref{control1} introduced in the present work allowed us to extend this probabilistic representation from the stochastic control case to the zero-sum game case and to relate these results with the classical theories of stochastic control and zero-sum SDG.\\ Let us also mention that a control problem similar to \eqref{control1} with a cost depending on the volatility of the martingale and arising from the asymptotic analysis of repeated games with incomplete information was analyzed in Gensbittel \cite{gensbittel}.

\subsection{About control problems with unbounded control spaces}

Control problems with square integrable controls (possibly unbounded) are considered for example in Krylov \cite{krylov} (see also his recent papers on games, e.g. \cite{krylovDPP}) and Touzi \cite{touzi}. More recently zero-sum SDG with unbounded controls and unbounded volatility have been studied  by Bayraktar-Yao \cite{BY}. However, none of these references deal with the associated equation with convexity constraints which characterizes the value of the zero-sum SDG we are studying. Indeed, if we formally write the standard upper and lower  Hamilton-Jacobi-Isaacs (HJI) equations in our model, assuming that the classical result applies, $V^-$ should be a viscosity supersolution of
\begin{equation}\label{supsolintro}
 -\frac{\partial V^-}{\partial t} (t,p,q)-H(t,p,q)- \sup_v \inf_u \frac{1}{2} Tr(M_{u,v} D^2V^-(t,p,q)) \geq 0,
\end{equation}
with $M_{u,v} = \begin{pmatrix} \sigma(p,u) \sigma(p,u)^t & 0 \\ 0 & \sigma(q,v) \sigma(q,v)^t \end{pmatrix}$, and similarly, $V^+$
should be a viscosity subsolution of
\begin{equation}\label{subsolintro}
-\frac{\partial V^+}{\partial t} (t,p,q)-H(t,p,q)-  \inf_u \sup_v \frac{1}{2} Tr(M_{u,v} D^2V^+(t,p,q)) \leq 0.
\end{equation}
Given a symmetric matrix $S$, the classical Isaacs' condition: 
\begin{equation}\label{isaacs}
 H(t,p,q)+\sup_u \inf_v \frac{1}{2} Tr(M_{u,v} S) = H(t,p,q)+\inf_v \sup_u \frac{1}{2} Tr(M_{u,v} S),
\end{equation}
does not always hold. Precisely, Isaac's condition holds with both sides of \eqref{isaacs} being finite if and only if 
\[S= \begin{pmatrix} S_1 & * \\ * & S_2 \end{pmatrix},\] 
with $\lambda_{\min}(p,S_1)\geq 0$ and $\lambda_{\max}(q,S_2) \leq 0$, and in this case both sides of \eqref{isaacs} are equal to $H(t,p,q)$. If $S$ does not fulfill these constraints, then both sides of \eqref{isaacs} are infinite and if $\lambda_{\min}(p,S_1)< 0$ and $\lambda_{\max}(q,S_2) > 0$, then the left-hand side of \eqref{isaacs} is $-\infty$ whereas the right-hand side is $+\infty$, so that \eqref{isaacs} does not hold even in a generalized sense.  
\pa
The variational characterization of the value function for our zero-sum SDG is -to our knowledge- not covered by any result in the  literature on SDG, in particular, one may not apply directly the results of Bayraktar-Yao \cite{BY} in the present context since the assumptions called $(A-u),(A-v)$ in \cite{BY} do not hold in the present model. 
Actually, we claim that the HJI equations (\ref{supsolintro},\ref{subsolintro}) do not hold in general and thus do not characterize the value function of our problem. This indicates that \eqref{edpintro} cannot be rewritten as a classical HJI equation, and is structurally different. This claim is proved rigorously through a very simple example detailed in subsection \ref{introconvex} below.


\subsection{The simple example of the convex envelope}\label{introconvex}
In order to give an easy and explicit illustration of our result, let us further simplify the model by assuming that   $J$ is reduced to a singleton and that $H$ does not depend on time.  In this case equation \eqref{edpintro2} becomes
\begin{equation}
\label{edpintro3}
\max\left\{ -\frac{\partial V}{\partial t} (t,p)-H(p);-\lambda_{min}(p,D^2_p V(t,p))\right\}=0.
\end{equation}
The unique solution is $V(t,p)= (T-t) Vex(H)(p)$, where $Vex(H)$ is the convex envelope of $H$, i.e. the largest convex function $f$ defined on $\Delta(I)$ such that $f \leq H$ (see \cite{c2} and \cite{cr2} for a detailed proof). 
Our result implies that $V$ is the value of the stochastic control problem with unbounded volatility \eqref{control1}, and thus we have:
\begin{equation}\label{control3}
(T-t) Vex(H)(p)=\inf_{u} \E[ \int_t^T H(X^{t,u,p}_s)ds ]. 
\end{equation}
This kind of representation for the convex envelope is not surprising at all and was probably already noticed by several authors (see e.g. \cite{cr2}, but also \cite{touzi} for a quite similar formulation with terminal cost). Moreover, a direct proof of \eqref{control3} is not difficult to obtain. However, note that this representation differs from the one suggested in Oberman \cite{oberman} for the convex envelope, which was a control problem with stopping and bounded volatility (let us also mention that convergence in long time to the convex envelope for stochastic control problems with bounded volatility was studied in \cite{carlier}). This example will serve us to show very easily that the function $V$ is not characterized through the following Hamilton-Jacobi equation that one could naively expect: 
\begin{equation}\label{solintro2}
- \frac{\partial V}{\partial t} (t,p)-H(p)- \inf_u \frac{1}{2} Tr(\sigma(p,u)\sigma(p,u)^t D^2V(t,p)) = 0.
\end{equation}
Indeed, $V$ is a subsolution of \eqref{solintro2} since for any test function $\phi \geq V$ such that $V(t,p)=\phi(t,p)$, the convexity of $V$ implies $\lambda_{\min}(p,D^2\phi(t,p)) \geq 0$, and thus
\[
 -\frac{\partial \phi}{\partial t} (t,p)-H(p)- \inf_u \frac{1}{2} Tr(\sigma(p,u)\sigma(p,u)^t D^2\phi(t,p)) = -Vex(H)(p)-H(p) \leq 0.
\]
However, one can see through a simple example that $V$ is not a supersolution of \eqref{solintro2}. Let $I=\{1,2\}$ and define $H(p)= \frac{1}{2} - |p-p_0|$ where $p_0=(\frac{1}{2},\frac{1}{2})$, so that $H$ is strictly concave, positive for $p$ in the relative interior of $\Delta(I)$ and equal to $0$ at points $(0,1)$ and $(1,0)$. It follows easily that $Vex(H)=0$ and one may thus use $\phi\equiv 0$   as a test function such that $\phi(t,p)\leq (T-t)Vex(H)(p)$ and $\phi(0,p_0)=V(0,p_0)$. We obtain:
\[
 -\frac{\partial \phi}{\partial t} (t,p_0)-H(p_0)- \inf_u \frac{1}{2} Tr(\sigma(p_0,u)\sigma(p_0,u)^t D^2\phi(t,p_0)) = -H(p_0) <0,
\]
which shows that $Vex(H)$ is not a supersolution. This very simple example shows that this kind of control problems with unbounded variance cannot be analyzed using classical Hamilton-Jacobi equations, but requires to consider equations with convexity constraints as \eqref{edpintro}.

\subsection{Simple pathwise strategies}\label{introstrategies}

The second contribution of this paper is to prove existence of the value for particular games with unbounded controls and unbounded volatility, having discontinuous coefficients, using simple pathwise strategies. The importance of using pathwise strategies was already outlined in Cardaliaguet-Rainer \cite{cr3}. Recall that the standard definition of a strategy introduced in Fleming-Souganidis \cite{FlemingSouganidis} (see also the definition in Buckdahn-Li \cite{BuckdahnLi} and in most of the papers on SDG) requires a player to react to the full control of his opponent. By full control, we mean the map $v(.)$ which associates to $\omega$ the control $v(\omega)$ (as we work on the canonical space here, $\omega$ denotes the Brownian trajectory). In many game modeling situations, one cannot require for a player to know what his opponent would have decided if another state of the world $\omega$ had occurred. Pathwise strategies are strategies which depend only on the actual realization $v(\omega)$ of the control. However, pathwise strategies introduced in Cardaliaguet-Rainer were universally measurable maps and their construction relied on the construction of Nutz \cite{Nutz}, hence on the axiom of the continuum (actually only on the weaker assumption that there exists a medial limit in the sense of Mokobodzki). The reason for this technical definition was the need of a pathwise version of the stochastic integral, in order for the players to be able to play simple feedback strategies. 

We chose to work here with a simpler definition, to avoid very technical measurability issues arising when trying to prove a dynamic programming principle. By requiring the players to play piecewise-constant controls on some intervals with rational endpoints, the controlled stochastic differential equations we consider are defined pathwise, which allow us to consider simple feedback strategies (actually quite close to the feedback strategies used in Pham-Zhang \cite{pham}), and also to use a measurable selection result.

\subsection{Open questions}

We are interested in several developments. As in the framework of repeated games developed by Laraki \cite{laraki}, replacing the simplices $\Delta(I)$ and $\Delta(J)$ by some compact convex sets $C,D$ give rise to a more general splitting game. On the other hand, in \cite{crrv} and \cite{gensbittel} the value of a continuous-time Markov game with incomplete information satisfies the following PDE with obstacles:
\[  \min\{ \max\{ -\frac{\partial V}{\partial t} -\LR(V)- u \,;\, -\lambda_{\min}(D^2_p V)\} \,;\, -\lambda_{\max}(D^2_q V) \}=0, \;  \]\\
Here the convexity constraints are the same as in the case we consider in the present paper, but the PDE is different. In particular it includes a drift term.
More generally, it is likely that, looking on more involved models, we have to appeal to the theory of viability. Indeed, the stochastic problem we introduce here can be seen as a game problem under constraints. However, the particular structure of the constraints (with affine borders) and of the dynamics (without drift term)  permits a very specific approach. Relaxing both assumptions, on the constraints and on the dynamics, this explicit treatment will not be possible anymore.

\section{The model}\label{model}

\subsection{Notations}

Let $T>0$ be a deterministic terminal time, and $I,J$ be
two non-empty finite sets which we will identify with the sets $\{ 1,\ldots,|I|\}$ and $\{1,...,|J|\}$ respectively. Let  $\Delta(I)=\{p \in \R^{|I|} \,|\, \forall i \in I, p_i \geq 0, \, \sum_{i \in I}p_i=1\}$ and $\Delta(J)$ denote the associated simplices. 

For all $0\leq t < t'\leq T$, we consider the {\bf Wiener space} $\Omega_{t,t'}=\{ \omega:[t,t']\rightarrow \R^{|I\times J|}$  continuous s.t. $\omega(t)=0\}$ endowed with the topology of uniform convergence and the associated Borel $\sigma$-algebra $\FR_{t,t'}$. The
Wiener measure on $\Omega_{t,t'}$ under which the canonical process $(B^{t,t'}(s,\omega)=\omega(s), s\in[t,t'])$ is a standard Brownian motion, is denoted by $\P_{t,t'}$. \\
Let $\F_{t,t'}^0=(\FR^{0,t'}_{t,s})_{s\in[t,t']}$ be the filtration generated by the coordinate process on $\Omega_{t,t'}$. 
We denote by $\F_{t,t'}=(\FR^{t'}_{t,s})_{s\in [t,t']}$ 
the smallest right-continuous filtration with respect to which the coordinate process is adapted and which contains all negligible sets for $\P_{t,t'}$. 
In the sequel, we have to decompose the canonical process $(B_s(\omega):=\omega(s),s\in[t,T])$ on $\Omega_t$ into $B_s=(B^{1}_s,B^{2}_s)$, where $B^{1}$ and $B^{2}$ are two independent Brownian motions with values in $\R^{|I|}$ (resp. $\R^{|J|}$). 

 The {\bf control spaces} are here  $U=\R^{|I\times I|}$, $V=\R^{|J\times J|}$. $U$ and $V$ are seen as spaces of matrices. 
On each euclidean space $\R^n$, we can define a bounded distance  by $d_b(x,y)=\frac{|x-y|}{1+|x-y|}$, topologically equivalent to the usual Euclidean distance. 
Let $U_{t,t'},V_{t,t'}$ be the sets of equivalence classes (with respect to the Lebesgue measure) of measurable maps from $[t,t']$ to $U,V$, endowed with the topology of convergence in measure. Note that this topology is metricized by the distance $d_1(\u,\u')=\int_{[t,t']} d_b(\u(s),\u'(s))ds$. 
Further let us introduce the set
 $\UR(t,t')$ of $\F_{t,t'}$ progressively measurable processes on $\Omega_{t,t'}$ taking values in $U$ and such that $\P_{t,t'}[ \int_{t}^{t'} |u_s|^2 ds <\infty]=1$, and the set $\VR(t,t')$ of $\F_{t,t'}$ progressively measurable processes on $\Omega_{t,t'}$ taking values in $V$ and such that $\P_{t,t'}[ \int_{t}^{t'} |v_s|^2 ds <\infty]=1$. 
In all these notations, we drop the index $t'$ if $t'=T$.

Finally, for two functions $\phi,\phi'$ from $[t,t']$ into a same space, we write $\phi\equiv \phi'$ on $[t,t']$, if $\phi(s)=\phi'(s)$ for Lebesgue-almost all $s\in[t,t']$.

\subsection{The stochastic differential equations}

\pa
\noi For fixed $(t,p)\in[0,T]\times\Delta(I)$ and $u\in\UR(t)$, we consider the following SDE
\begin{equation}
\label{sdeX}
X^{t,p,u}_s=p+\int_t^s\sigma(X^{t,p,u}_r,u_r)dB^1_r,\; s\in[t,T],
\end{equation}
where, for all $(x,u)\in\Delta(I)\times U$, $\sigma(x,u):= P_x u \in\R^{|I\times I|}$ where $P_x$ denotes the orthogonal projection on the tangent space of $\Delta(I)$ at $x$, $T_x(\Delta(I))$.
\pa
In the same way we introduce for $(t,q)\in[0,T]\times\Delta(J)$ and $v\in\VR(t)$ the SDE
\begin{equation}
\label{sdeY}
Y^{t,q,v}_s=q+\int_t^s\tau(Y^{t,q,v}_r,v_r)dB^2_r,\; s\in[t,T],
\end{equation}
where, for all $(y,v)\in\Delta(J)\times V$, $\tau(y,v)=P_y v \in\R^{|J\times J|}$ where $P_y$ denotes the orthogonal projection on the tangent space of $\Delta(J)$ at $y$.

\begin{Remark} 
Let us describe the structure of the tangent spaces of $\Delta(I)$.
For $p\in\Delta(I)$, define the support of $p$, $S(p):=\{i \in I \tq p_i>0\}$ and its complementary $E(p)=\{ i\in I \tq p_i=0 \}$. Then the tangent space $T_p(\Delta(I))$ depends only on $S(p)$. More precisely, 
it holds that 
\begin{equation}
\label{Tp}
 T_p(\Delta(I))=\{  y \in\R^{|I|}\tq \sum_{i\in I} y_i=0 \mbox{ and } \forall i\in E(p), y_i=0\}.
 \end{equation}
A useful consequence is that, if for $p,p'\in\Delta(I)$ we have $S(p)=S(p')$, then $T_p(\Delta(I))=T_{p'}(\Delta(I))$.
Furthermore we have an explicit formula for the orthogonal projection on the tangent space: 
For all vector $y\in\R^{|I|}$, 
\begin{equation}
\label{proj}
 (P_{p}y)_i= 
\left\{ \begin{matrix} 0 & \text{if} & i \notin S(p) \\ 
y_i - \frac{1}{|S(p)|} \sum_{i' \in S(p)} y_{i'} & \text{if} & i \in S(p)  \end{matrix}  \right.
\end{equation}
For any $I'\subset I$ and any $p\in\Delta(I)$ such that $S(p)=I'$, we write $P_{I'}$ for $P_p$.

\end{Remark}
\pa
Throughout the proofs, we use the convention $\inf\emptyset=+\infty$ when defining stopping times, although the time interval is $[t,T]$.
\pa

\begin{Proposition}
\label{prop1}
The SDE's (\ref{sdeX}) and (\ref{sdeY}) have unique strong solutions such that, $\P_t$-a.s. for all $s\in[t,T]$, $X_s\in\Delta(I)$ and $Y_s\in\Delta(J)$. 
\end{Proposition}

\begin{proof}
Because of the lack of regularity of $\sigma$ ($\sigma$ is not continuous in $x$!), we cannot use an existing theorem. We prove the result only for (\ref{sdeX}), because the arguments for (\ref{sdeY}) are the same. 
\pa
Set $p\in\Delta(I)$ and $I'=S(p)$. 
Up to restrict ourselves to the set $\Delta(I'):=\{ p'\in \Delta(I) | S(p')\subset I'\}$, we can suppose that $S(p)=I$. This is equivalent to say that $p$ belongs to the relative interior of $\Delta(I)$ denoted by $\mbox{Int}(\Delta(I))$. In this case, for all $u\in U^I$, $\sigma(p,u)=P_I u$.
\pa
Set $\tau^0:=t$ and consider the constant process $X^0_s\equiv p, s\in[t,T]$.\\
Suppose now that, for some $k\in\N$, an $\F_t$-stopping time $\tau^k\in[t,T]$ and an $\F_t$-adapted continuous process $X^k$ on $[t,T]$ are defined.
Set $I_k=S(X^k_{\tau^k})$.
Define the sequence of stopping times for $n \in \N$
\[ \theta^{k,n}:=\inf\{ s \in [\tau^k,T] | \int_{\tau^k}^s |P_{I_k} u_r|^2 dr \geq n\} \mbox{  and  } \theta^{k,\infty}=\sup_n\theta^{k,n} , \]
and the process
\[ X^{k+1}_s=X^k_{s\wedge\tau^k} +\int_{\tau^k}^{s\vee \tau^k} P_{I_k}u_r dB^1_r,\;  s \in[t,\theta^{k,\infty}) \cap [t,T],\]
where the integral is a local martingale.
We then extend (arbitrarily) the definition of $X^{k+1}$ by  $X^{k+1}_s=e_{i_0}$ for some $i_0\in I$, if  $s\geq \theta^{k,\infty}$ whenever $\theta^{k,\infty}\leq T$ so that $X^{k+1}$ is a well-defined c\`{a}dl\`{a}g process on $[t,T]$. 
\pa
Define $\tau^{k+1}:=\inf\{ s \in [t,T] | X^{k+1}_s \in \partial(\Delta(I_k))\}\wedge T$,
where $\partial(\Delta(I_k))$ is the relative boundary of $\Delta(I_k)$. Remark that, $\P_t$-a.s. on $\{ \tau^k<T\}$, we have $\tau^k<\tau^{k+1}$.
\pa
Let us prove that $\P_t[\theta^{k,\infty} \leq \tau^{k+1}]=0$.
Assume that $\P_t[\theta^{k,\infty} \leq \tau^{k+1}]>0$.  Note that on this event, the process $X^{k+1}$ stays in the relative interior  of $\Delta(I_k)$ during the time-interval  $[t,\theta^{k,\infty})$. Therefore, we have for all $n$
\[ 2 \geq \E_{t}[ |X^{k+1}_{\theta^{k,n} \wedge \tau^{k+1} }-X^k_{\tau^k }|^2  ]=\E_t\left[\int_{\tau^k}^{\theta^{k,n}\wedge \tau^{k+1} } |P_{I_k} u_s|^2 ds\right] \geq n \P_t[\theta^{k,\infty} \leq \tau^{k+1} ].\]  
This leads to a contradiction for large $n$. 
\pa
Now we set $I_{k+1}=S(X^{k+1}_{\tau^{k+1}})$. Since, on $\{ \tau^k<T\}$, $\Delta(I_{k+1})\subset\partial(\Delta(I_k))$, its dimension is  at most $\mbox{dim}(\Delta(I_k))-1$. 
Therefore, for all $k \geq |I|$, $\tau^k=T$ almost surely.
We set finally
\[ X^{t,p,u}_s=X^{|I|}_s, s\in[t,T].\]
It is easy to check that $X^{t,p,u}$ satisfies (\ref{sdeX}). Furthermore, by construction, we have $X_s^{t,p,u}\in\Delta(I), \P_t$-a.s for all $s\in[t,T]$.
\pa
Let us prove now the uniqueness: Let $\hat{X}$ denote another solution to \eqref{sdeX}. Define $\hat{\tau}^1=\inf\{  s \in [t,T] | \hat{X}_s \in \partial(\Delta(S(p)))\}\wedge T$, and note that for $s\in [t,T]$, we have on the event $\{s <\hat{\tau}^1\}$
\[ \hat{X}_s= p +\int_t^s P_p u_r dB^1_r= X^1_s,\]
which implies that $\hat{\tau}^1=\tau^1$, and thus that $\hat{X}_{\tau^1}=X^1_{\tau^1}$ since both processes have continuous trajectories. Define then $\hat{\tau}^2=\inf\{  s \in [t,T] | \hat{X}_s \in \partial(\Delta(S(\hat{X}_{\tau^1})))\}\wedge T$, and note that for $s\in [t,T]$, we have on the event $\{s < \hat{\tau}^2\}$
\[ \hat{X}_s= X^1_{s\wedge \tau^1} +\int_{\tau^1}^{s\vee \tau_1} P_{\hat{X}_{\tau^1}} u_r dB^1_r= X^2_s,\]
which implies that $\hat{\tau}^2=\tau^2$, and thus that $\hat{X}_{\tau^2}=X^2_{\tau^2}$. Proceeding by induction, and using that $\tau^{|I|}=T$ almost surely, we deduce that $\hat{X}_s=X^{|I|}_s$ for all $s \in [t,T]$. 
\end{proof}

\begin{Remark}
Using the same proof, \eqref{sdeX} and \eqref{sdeY} have unique strong solutions starting at $t' \in (t,T)$  with any random $\FR_{t,t'}$-measurable initial conditions taking values in $\Delta(I)\times \Delta(J)$ and controls $(u,v) \in \UR(t)\times \VR(t)$ (the solutions depend actually only on the restrictions on the time-interval $[t',T]$ of these controls).
\end{Remark}

\pa We shall use later a converse result, namely:
\begin{Proposition}
\label{Xu}
If, for some $p\in\Delta(I)$ and $u\in\UR(t)$, the process $X_s=p+\int_t^su_rdB^1_r, s\in[t,T]$ is well defined and $X_T$ belongs to $\Delta(I)$, then, $\P_t$-a.s., $u\equiv P_{X_\cdot}u$ on $[t,T]$ and $X_\cdot=X^{t,p,u}_\cdot$.
\end{Proposition}

\begin{proof} Since $(X_s)_{s\in[t,T]}$ is a martingale and $\Delta(I)$ a convex set, we have, $\P_t$-a.s. for all $s\in[t,T]$ $X_s\in\Delta(I)$. 
 Summing up all coordinates of $X_T$, we get 
\[ 1=\sum_{i\in I}X^i_T=1+\int_t^T\sum_{i\in I}u^i_rdB^1_r,\]
where for $i\in I$, $u^i \in \R^{|I|}$ denotes the $i$-th row of $u$. Therefore $\sum_i u^i\equiv 0$ $\P_t$-a.s. on $[t,T]$. Suppose further that, for some stopping time $\tau\in[t,T]$, some set $A\in\FR_{t,\tau}$ and some coordinate $i$, $X^i_\tau=0$ on $A$. Then, firstly,
 \[ X^i_{\tau}\ind_A=\E_t[X^i_T|\FR_{t,\tau}]\ind_A=0\] with $X^i_T\geq 0$, implies that $X^i_T=0$ on $A$, and,
  setting $\tau_A:=\tau\ind_A+T\ind_{A^c}$,  the relation
\[ X^i_T=X^i_{\tau_A}+\int_{\tau_A}^T u^i_rdB^1_r\]
 implies that $u^i\equiv 0$ $\P_t$-a.s. on $[\tau_A,T]$. The result follows by the explicit characterization (\ref{Tp}) of the tangent spaces $T_p$, $p\in\Delta(I)$. 
\end{proof}

\pa
Let $H:[0,T]\times\Delta(I)\times\Delta(J)\to\R$ be a bounded map, Lipschitz continuous in all its variables. Let $C>0$ be both a Lipschitz constant and an upper bound for $H$:
\[\begin{array}{c}
|H(t,p,q)-H(t',p',q')|\leq C(|t-t'|+|p-p'|+|q-q'|)\\
|H(t,p,q)|\leq C.
\end{array}\]

\pa
We set, for $(t,p,q,u,v)\in [0,T]\times\Delta(I)\times\Delta(J)\times\UR(t)\times\VR(t)$,
\[ J(t,p,q,u,v):=\E_t\left[\Int_t^T H(s,X^{t,p,u}_s,Y^{t,q,v}_s)ds\right].\]

The game is the following: Given the initial data $(t,p,q)$, Player 1 plays $u$ and wants to minimize $J(t,p,q,u,v)$, Player 2 plays $v$ and wants to maximize it.

We now introduce the controls and the strategies for the players.
As explained in the introduction, we restrict the players to use simple controls and simple strategies. On the one hand, this choice allows to have a pathwise definition of strategies, which means that a player reacts to the realization $u(\omega)$ of the control of his opponent without knowing the entire map $\omega \rightarrow u(\omega)$. Indeed, simple controls allow to define easily pathwise stochastic integrals. Another approach was initiated in \cite{cr3}, but this approach requires to assume the continuum hypothesis (which, by the result of Nutz \cite{Nutz}, implies the existence of a pathwise stochastic integral), and also to consider universally measurable strategies, which would make the proofs of the dynamic programming principle highly technical.

\pa 
We denote by $\UR^s(t)$ (resp.  $\VR^s(t)$) the set of piecewise constant controls with rational grid, i.e. $u\in \UR^s(t)$ if there exist $t=t_0<...<t_m=T$ and (Borel)-measurable maps $g_j$ from $\Omega_{t,t_j}$ to $U$ (resp. $V$) for $j=0,...,m-1$ such that $t_1,...,t_{m-1}$ belong to $\Q \cap (t,T)$ and
\[u(\omega, s)=\sum_{j}\ind_{[t_j,t_{j+1})}(s)g_j(\omega|_{[t,t_j]}).\]
These controls will be called {\bf simple controls}. 
\pa
Similarly, let $U^s_t\subset U_t$ (resp. $V^s_t \subset V_t$) denote the subset of piecewise constant trajectories with rational grid.  Generic trajectories will be denoted $\u,\v$ in contrast to controls $u,v$.
\pa

\begin{Definition}
\label{admissible}
\item An  {\bf admissible}  strategy for Player 1 at time $t$ is a Borel map $\alpha:\Omega_t\times V^s_t\to U^s_t$ such that there exist a sequence $t=t_0<t_1<...<t_m=T$  such that $t_1,...,t_{m-1}$ belong to $\Q \cap (t,T)$ and
\[ \alpha(\omega,\v)(s)= \sum_{j=0}^{m-1}\ind_{[t_j,t_{j+1})}(s)\alpha^j((\omega,\v)|_{[t,t_j]}),
\]
 for some Borel maps $\alpha^j:\Omega_{t,t_j}\times V^s_{t,t_j}\rightarrow U$.
We denote by $\AR(t)$ the set of admissible strategies for Player 1.\\
The set of admissible strategies for Player $2$, denoted by $\BR(t)$ is defined similarly.

\end{Definition}

\pa Remark that these strategies are pathwise strategies with delay. They are pahtwise strategies because they depend on the realization of the control of the opponent $v(\omega)$ and not on the whole map $\omega \rightarrow v(\omega)$. They are also strategies with delay in the spirit of e.g. \cite{cr2} or \cite{oliu}:  it is easy to see that, given the associated sequence  $t=t_0<t_1<...<t_m=T$, for any $j\in\{ 0,\ldots,m-1\}$,   the answer of player 1 on the time interval $[t,t_{j+1}]$  depends on the Brownian path and the action of player 2  only through their restriction to the time interval $[t,t_j]$.
\pa
The next Lemma is stated without proof and easy to verify.
\begin{Lemma}\label{measurabilityv2}
For all $\alpha \in \AR(t)$ and $v \in \VR^s(t)$, there exists a process $\alpha(v) \in \UR^s(t)$ such that for all $\omega \in \Omega_t$, 
\[ \alpha(v)(\omega,.)\equiv \alpha(\omega,v(\omega,.)).\]
\end{Lemma} 
\pa
The following lemma is standard and can be easily adapted from the existing literature, as for instance in \cite{cr1}). 
\begin{Lemma}\label{fixedpoint} For all $t\in[t,T]$, for all $(\alpha,\beta)\in\AR(t)\times\BR(t)$, there exists a unique pair of simple controls $(u,v)\in \UR^s(t)\times\VR^s(t)$ which satisfies, $\P_t$-a.s.
\begin{equation}
\label{fix}
 u\equiv\alpha(v),\;  v\equiv\beta(u). 
\end{equation}
We denote them by $(u^{\alpha,\beta},v^{\alpha,\beta})$. 
\end{Lemma}

We are ready now to define the lower and upper value functions for the game:
\[ V^+(t,p,q)=\inf_{\alpha\in\AR(t)}\sup_{\beta\in \BR(t) }J(t,p,q,u^{\alpha,\beta},v^{\alpha,\beta}).\]
\[ V^-(t,p,q)=\sup_{\beta\in\BR(t)}\inf_{\alpha \in \AR(t)}J(t,p,q,u^{\alpha,\beta},v^{\alpha,\beta}).\]
Note that, as usual,  $V^- \leq V^+$ and that 
\[ V^+(t,p,q)=\inf_{\alpha\in\AR(t)}\sup_{v \in \VR^s(t) }J(t,p,q,\alpha(v),v),\]
\[ V^-(t,p,q)=\sup_{\beta\in\BR(t)}\inf_{u\in \UR^s(t)}J(t,p,q,u,\beta(u)).\]

\pa
We need to introduce some notations before stating the main theorem.
Let $\SR_I$ denote the set of $I\times I$ symmetric matrices. For $(p,A)\in\Delta(I)\times\SR_I$, define
\[\lambda_{min}(p,A):= \min_{z\in T_p(\Delta(I))\setminus \{ 0\} }\langle Az,z\rangle /|z|^2, \]
with the convention $\min(\emptyset)=+\infty$. 
Similarly, for $(q,B)\in\Delta(J)\times\SR_J$, define 
\[\lambda_{max}(q,B):=\max_{z\in T_q(\Delta(J))\setminus\{ 0\}}\langle Bz,z\rangle /|z|^2, \]
with the convention $\max(\emptyset)=- \infty$. 
\pa 
The main theorem of this paper is the following.

\begin{Theorem}\label{maintheorem}
The game has a value $V:=V^+=V^-$ which is the unique Lipschitz  viscosity-solution of the following barrier equation: 
\begin{equation}
\label{hji}
\min\left\{\max\left\{ -\frac{\partial V}{\partial t} (t,p,q)-H(t,p,q);-\lambda_{min}(p,D^2_p V(t,p,q))\right\};-\lambda_{max}(q,D^2_q V(t,p,q))\right\}=0,
\end{equation}
where $D^2_p V,D^2_q V$ the second-order derivatives with respect to $p$ and $q$.
\end{Theorem}
\pa
See section \ref{sectionviscosity} for the precise definition of viscosity solution of \eqref{hji}.

\section{Regularity of the value functions.}

The aim of what follows is to prove that the value functions are Lipschitz in all their variables.

\begin{Lemma}
\label{th}
Let $(t,p,q)\in[0,T)\times\Delta(I)\times\Delta(J)$ and $h\in(0,T-t]$. 
\begin{enumerate}
\item For all $\alpha\in\AR(t+h)$ and $v\in\VR^s(t)$, there exists $\tilde\alpha\in\AR(t)$ and $\tilde v\in\VR^s(t+h)$ such that
\[ |J(t,p,q,\tilde\alpha(v),v)-J(t+h,p,q,\alpha(\tilde v),\tilde v)|\leq 8Ch,\]
\item for all $\alpha\in\AR(t)$ and $v\in\VR^s(t+h)$, there exists $\tilde\alpha\in\AR(t+h)$ and $\tilde v\in\VR^s(t)$ such that
\[ |J(t+h,p,q,\tilde\alpha(v),v)-J(t,p,q,\alpha(\tilde v),\tilde v)|\leq 8Ch.\]
\end{enumerate}
\end{Lemma}
\begin{proof}
The proof looks very involved but it is not: based on the scaling property of the Brownian motion and the linearity of $\sigma$ and $\tau$, it proceeds by very elementary transformations. We only prove the first point as the second one follows by symmetry.\\
Let $\alpha\in\AR(t+h)$ and $v\in\VR^s(t)$ and let $t+h=t_0<t_1<...<t_m=T$ be the grid associated to $\alpha$ and $t=t'_0<t'_1<....<t'_n=T$ the grid associated to $v$. Let $\eta \in [t+2h,t+3h]\cap \Q$ and $\phi:[t+h,T]\rightarrow [t,T]$ defined by
\[ \phi(s)=\left\{
\begin{array}{ll}
\psi(s),&\mbox{ if } s\in[t+h,\eta),\\
s,&\mbox{ if } s\in[\eta,T].
\end{array}\right.,\]
where $\psi:[t+h,\eta]\rightarrow [t,\eta]$ is an increasing homeomorphism which can be chosen such that
\begin{itemize}
\item $\phi$ is piecewise-affine on the partition $t+h=s_0<s_1<...<s_{N}=T$, with 
\[\{ s_1,\ldots,s_{N-1}\}=\{\eta\} \cup \Big(\left(\{t_j,j=1,...,m-1\} \cup \{ \phi^{-1} (t'_j), j=1,...,n-1\} \right) \cap (t+h,T) \Big).\] 
\item the image of $\{t_1,....,t_{m-1}\} \cap [t+h,\eta]$ by $\psi$ belongs to $\Q$,
\item the image of $\{t'_1,....,t'_{n-1}\} \cap [t,\eta]$ by $\psi^{-1}$ belongs to $\Q$,
\item for all $r,s \in [t+h,\eta]$, $|r-s|  \leq |\psi(r)-\psi(s)|  \leq 3 |r-s|$.
\end{itemize}
The proof that such a map $\psi$ exists is left to the reader, the idea is to slightly perturb the map $s \rightarrow t+ \frac{\eta-t}{\eta - (t+h)}(s-(t+h))$ on $[t+h,\eta]$.  
Remark that $\phi$ is an increasing homeomorphism from $[t+h,T]$ to $[t,T]$ and that the image of $\{t_1,....,t_{m-1}\}$ by $\phi$ belongs to $\Q$ and the image of $\{t'_1,....,t'_{n-1}\}$ by $\phi^{-1}$ belongs to $\Q$.
\pa
Let us define the map $R_\phi : \Omega_t \rightarrow \Omega_{t+h}$ by:
\[ \forall s\in[t+h,T],\; R_\phi(\omega)(s):= \sum_{i=0}^{N-1} \frac{1}{\sqrt{ \phi'(s_i)}} \Big( \omega(\phi(s \wedge s_{i+1}))- \omega(\phi(s_i)) \Big) \ind_{s \geq s_i},\]
where $\phi'(s_i)$ denote the right-derivative of $\phi$ at $s_i$.
\pa
We get a $\R^{|I\times J|}$-valued, standard Brownian motion $(\tilde{B}_s)_{s \in [t+h,T]}$ on $\Omega_t$ by setting:
\[ \forall \omega \in \Omega_t, \forall s\in [t+h,T], \; \tilde{B}(\omega)(s)= R_\phi(\omega)(s). \]
\pa Define the process $\tilde{v}\in \VR^s(t+h)$ by
\[ \forall (\omega,s)\in \Omega_{t+h}\times[t+h,T], \; \tilde{v}(\omega,s) :=\sum_{i=0}^{N-1} 
\sqrt{\phi'(s_i)} v(R_{\phi}^{-1}(\omega),{\phi(s)}) \ind_{s\in[s_i, s_{i+1})} .\]
\pa
For $s\in[t+h,T]$, set $\bar Y_s:=Y^{t,q,v}_{\phi(s)}$. The process $\bar Y$ satisfies:
\begin{align*}
\bar Y_s &= q+\int_t^{\phi(s)}\tau(Y^{t,q,v}_r,v_r)dB^{2}_r\\
&= q+\int_{t+h}^s\tau(\bar Y_r,v_{\phi(r)})dB^{2}_{\phi(r)}=q+\int_{t+h}^s P_{\bar Y_r} v_{\phi(r)} dB^{2}_{\phi(r)}\\
&=q+\int_{t+h}^s P_{\bar Y_r} \tilde{v}(\tilde{B}(\omega),r) d \tilde{B}^2_r =q+\int_{t+h}^s\tau(\bar Y_r,\tilde{v}(\tilde{B}(\omega),r))d\tilde B^2_r,
\end{align*}
\pa Define also the one-to-one continuous mappings:
\[ T_\phi : V_t \rightarrow V_{t+h}, \; T_\phi(\v)(s)= \sum_{i=0}^{N-1} 
\sqrt{\phi'(s_i)} \v({\phi(s)}) \ind_{[s_i, s_{i+1})}(s).\]
\[ T'_\phi : U_{t+h} \rightarrow U_{t}, \; T'_\phi(\u)(s)=\sum_{i=0}^{N-1}\frac{1}{\sqrt{\phi'(s_i)}} \u(\phi^{-1}(s)){\ind_{[\phi(s_i), \phi(s_{i+1}))}(s)}.\]
Using these notations, let us define  a strategy $\tilde\alpha:\Omega_t\times V^s_t\rightarrow U_t$ by:
\[ \tilde\alpha(\omega,\v)(s)= (T'_\phi)(\alpha(R_\phi(\omega),T_\phi(\v)))(s).\]
Let us show that $\tilde\alpha$ belongs to $\AR(t)$: since, by assumption the time grid $t+h=t_0<\ldots<t_m=T$ associated to $\alpha$ is  a subset of $t+h=s_0<\ldots<s_N=T$, we may rewrite $\alpha$ as:
\[ \alpha(\omega,\v)(s)= \sum_{i=0}^{N-1}\ind_{[s_i,s_{i+1})}(s)\alpha^i((\omega,\v)|_{[t+h,s_i]}),
\]
with $\alpha^i$ Borel-measurable from $\Omega_{t+h,s_i}\times V_{t+h,s_i}$ to $U$.
Therefore $\tilde\alpha$ may be reformulated as
\[ \tilde\alpha(\omega,\v)(s)=\sum_{i=0}^{N-1}\frac{1}{\sqrt{\phi'(s_i)}}\alpha^i\left((R_\phi(\omega),T_\phi(\v))|_{[t+h,s_i]}\right)\ind_{[\phi(s_i), \phi(s_{i+1}))}(s)\]
and satisfies clearly the assumptions of an admissible strategy given in Definition \ref{admissible}, related to the time grid $t=\phi(s_0)<\ldots<\phi(s_N)=T$ since $(R_\phi(\omega),T_\phi(\v))|_{[t+h,s_i]}$  is by construction a measurable map of $(\omega,\v)|_{[t,\phi(s_i)]}$.
\pa Furthermore, for $\P_t\otimes ds$ almost every $(\omega,s)$, it holds that: 
\begin{align*}
 \tilde{\alpha}(v)(\omega,s) &= \sum_{i=0}^{N-1}\frac{1}{\sqrt {\phi'(s_i)}} \alpha(\tilde{B}(\omega), \tilde{v}(\tilde{B}(\omega),.))(\phi^{-1}(s)){ \ind_{[\phi(s_i), \phi(s_{i+1}))}(s)} \\
&=\sum_{i=0}^{N-1}\frac{1}{\sqrt {\phi'(s_i)}} \alpha(\tilde{v})(\tilde{B}(\omega),.)(\phi^{-1}(s)){ \ind_{[\phi(s_i), \phi(s_{i+1}))}(s)}.
\end{align*}
Then, for $s\in[t+h,T]$, writing $\bar X_s:=X^{t,p,\tilde\alpha(v)}_{\phi(s)}$, we have
\begin{align*}
\bar X_s &= p +\int_t^{\phi(s)}\sigma( X^{t,p,\tilde\alpha(v)}_r,\tilde\alpha(v)_r)dB^{1}_r=p +\int_{t+h}^{s} P_{\bar{X}_r} \tilde\alpha(v)_{\phi(r)} dB^{1}_{\phi(r)}\\
&=p+\int_{t+h}^s P_{\bar{X}_r} \alpha(\tilde{v})(\tilde{B}(\omega),r) d\tilde B^1_r=p+\int_{t+h}^s\sigma(\bar X_r,\alpha(\tilde{v})(\tilde{B}(\omega),r))d\tilde B^1_r.
\end{align*}
Using the uniqueness of strong solutions for the system of SDE (\ref{sdeX},\ref{sdeY}), the pairs $(\bar X,\bar Y)$ and 
\[ (X,Y):=(X^{t+h,p,\alpha(\tilde v)},Y^{t+h,q,\tilde v})\]
have the same law. The same change of variables as above leads us to
\[\begin{array}{rl}
J(t,p,q,\tilde\alpha,v)=&
\E_t[\int_{t+h}^T H(\phi(s),\bar{X}_{s},\bar{Y}_{s})d\phi(s)]\\
=&\E_t[\int_{t+h}^{\eta} H(\phi(s),X_s,Y_s)d\phi(s)]
+\E_t[\int_{\eta}^T H(s,X_s,Y_s)ds]\\
=&A(h)+J(t+h,p,q,\alpha,\tilde v),
\end{array}\] 
with 
\[A(h):=\E[\int_{t+h}^{\eta}H(\phi(s),X_s,Y_s)d\phi(s)]-\E[\int_{t+h}^{\eta} H(s,X_s,Y_s)ds].\] Due to the boundedness of $H$ and using that $|\phi'|_\infty \leq 3$ and $\eta \in [t+2h,t+3h]$, we have
\[ |A(h)|\leq 8\|H\|_\infty h,\] and the result follows. 
\end{proof}

\begin{Proposition}
The value functions $V^+,V^-$ are Lipschitz in $t$ : for all $(t,p,q)\in[0,T]\times\Delta(I)\times\Delta(J)$ and $h\in [0,T-t]$,
\[ |W(t+h,p,q)-W(t,p,q)|\leq 8Ch,\] 
for $W=V^+,V^-$.
\end{Proposition}

\begin{proof} For $W=V^+$, the proposition follows directly from Lemma \ref{th}. For $W=V^-$, we use  the symmetric result to Lemma \ref{th} where $(\tilde\alpha(v),v)$ is replaced by $(u,\tilde\beta(u))$ and $(\alpha(\tilde v),\tilde v)$ by $(\tilde u,\beta(\tilde u))$.
\end{proof}

\pa The Lipschitz continuity of the values with respect to $p$ and $q$ is based on the following proposition.

\pa
\begin{Proposition} 
\label{lip}
(Lipschitz continuity in $p$)
Let $t\in[0,T]$, $u\in\UR(t)$ and $p,\bar p\in\Delta(I)$.
Then there exists some constant $\bar C>0$ which depends only on $|I|$, such that, for all $s\in[t,T]$,
\[ \E_t\left[|X^{t,p,u}_s-X^{t,\bar p,u}_s|\right]\leq \bar C|p-\bar p|.\]
\end{Proposition}

\begin{proof}
Recall the notation
$S(p)=\{ i\in I, p_i>0\}$ and 
let us define $A_s=|S(X^{t,p,u}_s)|+|S(X^{t,\bar p,u}_s)|, s\in[t,T]$. $A$ is non-increasing, c\`adl\`ag, integer-valued process. Set $\tau^0=t$ and $\tau^{k+1}:= \inf \{ s\geq \tau^k | A_s <A_{\tau^k}\}\wedge T$. As in the proof of Theorem \ref{prop1}, this defines a non-decreasing sequence of stopping times which, at least for $k\geq 2|I|-1$ is constant equal to $T$.\\
Remark that, for all $t\leq s\leq T$, on $ \{ s<\tau^1\}$, $S(X^{t,p,u}_s)=S(p), \; S(X^{t,\bar p,u}_s)=S(\bar p)$ and  therefore $P_{X^{t,p,u}_s}=P_p$ and  $P_{X^{t,\bar p,u}_s}=P_{\bar p}$. \\
Set 
\[ \xi:=X^{t,p,u}_{\tau^1}=p+\int_t^{\tau^1}P_pu_sdB^1_s \;\mbox{ and }\;
\bar\xi:=X^{t,\bar p,u}_{\tau^1}=\bar p+\int_t^{\tau^1}P_{\bar p}u_sdB^1_s.\]
\pa We have to compute $\E_t[|\xi_i-\bar\xi_i|]$ for all $i\in I$:
\pa If $i\notin S(p)\cup S(\bar p)$,
\begin{equation}
\label{i1}
 \E_t[|\xi_i-\bar\xi_i|]=0=|p_i-\bar p_i|.
\end{equation}
\pa
If $i \in S(p)\setminus S(\bar p)$, 
$|\xi_i-\bar\xi_i| = \xi_i= p_i+ \sum_{j=1}^n \int_t^{\tau^1} (P_{p}u_s)_{i,j} dB^{1,j}_s$
and therefore 
\begin{equation}
\label{i2}
\E_t[|\xi_i-\bar\xi_i|]=p_i=|p_i-\bar p_i|.
\end{equation}
\pa In the same way, if $i \in S(\bar p)\setminus S(p)$, then 
\begin{equation}
\label{i3}
 \E_t[|\xi_i-\bar\xi_i|]=\bar{p}_i=|p_i-\bar p_i|.
 \end{equation}
\pa Finally define $I_1= S(p)\cap S(\bar{p})$. The explicit projection-formula  (\ref{proj}) leads to the following: for all $i\in I_1$,
\[\xi_i- \bar\xi_i= p_i- \bar{p}_i -\int_t^s \langle  \left( \frac{1}{|S(p)|}\sum_{i'\in S(p)}(u_s)_{i',j} -\frac{1}{|S(\bar{p})|}\sum_{i'\in S(\bar{p})}(u_s)_{i',j} \right)_{j \in I}  , dB^1_r   \rangle. \]
This implies that for all $i' \in I_1$, 
\[ \xi_i- \bar\xi_i -(p_i-\bar{p}_i)=\xi_{i'}- \bar\xi_{i'}-(p_{i'}-\bar{p}_{i'}).\]
If $I_1\neq \emptyset$, using this equality together with the fact that $\sum_{i'\in I} (\xi-\bar\xi)_{i'}=0$, we  deduce that, for all $i\in I_1$:
\begin{align*}
|\xi_i- \bar\xi_i| &\leq | p_i-\bar{p}_i|+ | \xi_i- \bar\xi_i -(p_i-\bar{p}_i)| \\
&= | p_i-\bar{p}_i|+\frac{1}{|I_1|}\left| \sum_{i' \in I_1}\left(\xi_{i'}- \bar\xi_{i'} -(p_{i'}-\bar{p}_{i'})  \right) \right| \\
&\leq 2| p-\bar{p}|+\frac{1}{|I_1|}\left| \sum_{i' \in I_1}\left(\xi_{i'}- \bar\xi_{i'} \right)\right| \\
& \leq  2|p-\bar{p}|+ \frac{1}{|I_1|}\sum_{i' \in I\setminus I_1} |\xi_{i'}- \bar\xi_{i'}| \\
& =  2|p-\bar p|+ \frac{1}{|I_1|}\left( \sum_{i' \in S(p)\setminus S(\bar{p})}\xi_{i'} +  \sum_{i' \in S(\bar p)\setminus S(p)} \bar\xi_{i'} \right) .
\end{align*}
We have 
\begin{align*}
 \E_t\left[ \frac{1}{|I_1|}\left( \sum_{i' \in S(p)\setminus S(\bar{p})}\xi_{i'} +  \sum_{i' \in S(\bar p)\setminus S(p)} \bar\xi_{i'} \right) \right]&=  \frac{1}{|I_1|}\left( \sum_{i' \in S(p)\setminus S(\bar{p})}|p_{i'}-\bar{p}_{i'}| +  \sum_{i' \in S(\bar p)\setminus S(p)} |p_{i'}-\bar{p}_{i'}| \right)  \\
 &\leq \sqrt{|I|} |p-\bar{p}|.
 \end{align*}
It follows that, for $i\in I_1$,
\begin{equation}
\label{i4}
\E_t[|\xi_i- \bar\xi_i|] \leq (2+\sqrt{|I|}) |p-\bar{p}|.
\end{equation} 
Putting together (\ref{i1})-(\ref{i4}), we get
\[\E_t[|\xi- \bar\xi|] \leq (2+\sqrt{|I|})|I||p-\bar p|.\]
By induction, and using the same arguments as above, we have for all $k$
\[ \E_t[|X^{t,p,u}_{\tau^{k+1}}- X^{t,\bar p,u}_{\tau^{k+1}}| \,|\, \FR_{t,\tau^k} ] \leq (2+\sqrt{|I|})|I| |X^{t,p,u}_{\tau^{k}}- X^{t,\bar p,u}_{\tau^{k}}|.\]
However, for $k \geq 2|I|-1$, we have $\tau^k=T$, so that 
\[ \E_t[|X^{t,p,u}_{T}- X^{t,\bar p,u}_{T}|] \leq ((2+\sqrt{|I|})|I|)^{2|I|-1}|p-\bar p|.\]
Using Jensen's inequality, we deduce finally that:
\[ \forall s\in[t,T], \;\E_t[|X^{t,p,u}_{s}- X^{t,\bar p,u}_{s}|] \leq ((2+\sqrt{|I|})|I|)^{2|I|-1}|p-\bar p|.\]
\end{proof}

\begin{Proposition}
The value functions $V^+,V^-$ are $C\bar CT$-Lipschitz in $p$ and in $q$.
\end{Proposition}

\begin{proof} The Lipschitz continuity in $p$ follows classically from the proposition \ref{lip}. And by the symmetric roles played by $p$ and $q$, the Lipschitz continuity in $q$ also follows.
\end{proof}

\section{Dynamic programming}

\noi The standard procedure would be to prove now the properties of convexity/concavity  of the value functions. The reason why we start with the dynamic programming principle (DPP) is that its proof is very standard, while the convexity/concavity are not only harder to establish, but borrow also some of the techniques developed for the  DPP, which, by this way, will be firstly exposed in this familiar setting.

\begin{Proposition}\label{propdpp}
Let $(t,p,q)\in [0,T)\times\Delta(I)\times\Delta(J)$ and $h>0 $ such that $t+h \in (t,T] \cap \Q$.
Then it holds that
\begin{equation}
\label{dpp}
V^-(t,p,q)\geq \sup_{\beta\in\BR(t)}\inf_{u\in\UR^s(t)} \E_t\left[
\Int_t^{t+h} H(s,X^{t,p,u}_s,Y^{t,q,\beta(u)}_s)ds+ V^-(t+h,X^{t,p,u}_{t+h},Y^{t,q,\beta(u)}_{t+h})\right].
\end{equation}
\begin{equation}
\label{dpp2}
V^+(t,p,q)\leq \inf_{\alpha\in\AR(t)}\sup_{v\in\VR^s(t)} \E_t\left[
\Int_t^{t+h} H(s,X^{t,p,\alpha(v)}_s,Y^{t,q,v}_s)ds+ V^+(t+h,X^{t,p,\alpha(v)}_{t+h},Y^{t,q,v}_{t+h})\right].
\end{equation}
\end{Proposition}
\begin{proof}
We only prove \eqref{dpp}, the proof of \eqref{dpp2} being similar.
\pa 
We denote by $RHS$ the right-hand side of equation \eqref{dpp}.
For $\varepsilon>0$, let $\beta^0\in\BR(t)$ be $\varepsilon$-optimal for $RHS$. 
We can find $(O_1,\ldots,O_M)$ a measurable partition of $\Delta(I)\times\Delta(J)$ with diameter smaller than  $\varepsilon$. For each $m\in\{ 1,\ldots, M\}$, we pick some couple $(p^m,q^m)\in O_m$ and choose $\beta^m\in\BR(t+h)$ which is $\varepsilon$-optimal for $V^-(t+h,p^m,q^m)$. 
Using Lemma \ref{SDE-borel}, we have almost surely
\[(X^{t,p,\u}_{t+h},Y^{t,q,\beta^0(\u)}_{t+h})= (\Phi^X_{t,t+h}(p,\omega|_{[t,t+h]},\u|_{[t,t+h]})(t+h),\Phi^Y_{t,t+h}(q,\omega|_{[t,t+h]},\beta^0(\u)|_{[t,t+h]})(t+h)),\]
where $(\Phi^X_{t,t+h},\Phi^Y_{t,t+h})$ are  Borel maps. In the following we will identify $(X^{t,p,\u}_{t+h},Y^{t,q,\beta^0(\u)}_{t+h})$ with these Borel maps to simplify notations.  Using this convention and the fact that $\beta^0(\u)|_{[t,t+h]}$ is a measurable function of $\u|_{[t,t+h]}$, we can define a new strategy $\tilde\beta$ by setting, for all $(\omega,\u)\in \Omega_t\times U^s_t$,
\[\tilde\beta(\omega,\u)(s)=\left\{\begin{array}{ll}
\beta^0(\omega,\u)(s),& \mbox{ if }t\leq s<t+h,\\
\beta^m(\omega|_{[t+h,T]}-\omega(t+h),\u|_{[t+h,T]})(s),&\mbox{ if } (s,\omega)\in[t+h,T]\\
& \mbox{ and }(X^{t,p,\u}_{t+h},Y^{t,q,\beta^0(\u)}_{t+h})(\omega)\in O_m.
\end{array}
\right.\]
Let us show that it is also admissible in the sense of  Definition \ref{admissible}: 
Let $\pi_1,...,\pi_M$ denote the grid associated to $\beta^1,\ldots,\beta^M$ with $\pi_m=\{t+h=t^m_0<...<t^m_{N_m}=T\}$. We consider a time grid $\pi=\{ t_0=t<\ldots<t_N\}$  which coincides on $[t,t+h)$ with the time grid associated to $\beta^0$ and contains the time grids $\pi_1,...,\pi_m$.
Then we can write
\[ \tilde\beta(\omega,\u)=\sum_{j=0}^{N-1}\ind_{[t_j,t_{j+1})}(s)\tilde\beta^j((\omega,\u)|_{[t,t_j]}),\]
with, for all $(\omega,\u)\in\Omega_{t,t_{j+1}}\times U^s_{t,t_{j+1}}$,
\[ \tilde\beta^j(\omega,\u)=\sum_{m=1}^M\ind_{O_m}(X^{t,p,\u}_{t+h},Y^{t,q,\beta^0(\u)}_{t+h})\ind_{[t^m_{k(m,j)},t^m_{k(m,j)+1})}(s)\beta^{m,k(m,j)}((\omega|_{[t+h,t_{k(m,j)}]}-\omega(t+h),\u|_{[t+h,t_{k(m,j)}]}),\]
where
$\beta^m(\omega,\u)(s)=\sum_{k=0}^{N_m-1}\ind_{[t^m_k,t^m_{k+1})}(s)\beta^{m,k}((\omega,\u)|_{[t+h,t^m_k]})$ is the decomposition of $\beta^m$ and $k(m,j)$ is the unique integer such that $[t_j,t_{j+1}) \subset [t^m_{k(m,j)},t^m_{k(m,j)+1})$.
\pa 
Let us now use the identification $\Omega_t=\Omega_{t,t+h}\times\Omega_{t+h}$ with $\omega=\omega_1\otimes\omega_2$, where $\omega_1:=\omega|_{[t,t+h]}$ and $\omega_2:=\omega|_{[t+h,T]}-\omega(t+h)$.
Let us fix $u\in\UR^s(t)$. Then, for all $\omega_1\in\Omega_{t,t+h}$, the map 
\[ (s,\omega_2)\in[t+h,T]\times\Omega_{t+h}\rightarrow u(\omega_1\otimes\omega_2,s)\]
 defines an element of $\UR^s(t+h)$ denoted by $u(\omega_1)$. It follows from Lemma \ref{SDE-borel} that $X^{t,p,u}_{t+h}$ is almost surely equal to a measurable function of $\omega_1$ denoted $X(\omega_1)$ and that we can write
\begin{equation}
\label{X}
  \forall s \in [t+h,T], \; X^{t,p,u}_s(\omega_1\otimes \omega_2)= X^{t+h,X(\omega_1),u(\omega_1)}_s(\omega_2),
  \end{equation}
In the same way, $Y^{t,q,\beta^0(u)}_{t+h}(\omega)=Y(\omega_1)$ almost surely for some Borel map $Y(\omega_1)$.
Then the sets $A_m=\{ \omega_1\in\Omega_{t,t+h}, (X,Y)(\omega_1)\in O_m\}$ are well defined and form a partition of $\Omega_{t,t+h}$. Moreover we have, on each $A_m$,
\begin{equation}
\label{Y}
\forall s \in [t+h,T], \; Y^{t,q,\tilde\beta(u)}_s(\omega_1\otimes \omega_2)= Y^{t+h,Y(\omega_1),\beta^m(u(\omega_1))}_s(\omega_2).
\end{equation}
Note that the right-hand sides of (\ref{X}) and (\ref{Y}) are $\P_t$-almost surely equal to jointly measurable functions in $(\omega_1,\omega_2)$ (see Lemma \ref{SDE-borel}).
Proposition \ref{lip} implies that, for all $s\in[t+h,T]$, for all $m\in\{ 1,\ldots,M\}$  and $\omega_1\in A_m$,
\[
\E_{t+h}[|X^{t+h,X(\omega_1),u(\omega_1)}_s-X^{t+h,p^m,u(\omega_1)}_s|]\leq \bar C\varepsilon, \]
\[\E_{t+h}[|Y^{t+h,Y(\omega_1),\beta^m(u(\omega_1))}_s-Y^{t+h,q^m,\beta^m(u(\omega_1))}_s|]\leq \bar C \varepsilon.
\]
It follows that, setting $\tilde C=C\bar CT$, we have :
\begin{align*}
J(t,p,q,u,\tilde{\beta})\geq & \E_t\left[\Int_t^{t+h}H(s,X^{t,p,u}_s,Y_s^{t,q,\beta^0(u)})ds\right] +\\
 \Int_{\Omega_{t,t+h}}\sum_m\ind_{A_m}(\omega_1)&\E_{t+h}\left[\Int_{t+h}^TH\left(s,X_s^{t+h,p^m,u(\omega_1,\cdot)},Y^{t+h,q^m,\beta^m(u(\omega_1))}_s\right)ds\right]d\P_{t,t+h}(\omega_1)- 2\tilde C\varepsilon\\
\geq &\E_t\left[\Int_t^{t+h}H(s,X^{t,p,u}_s,Y_s^{t,q,\beta^0(u)})ds+\sum_m\ind_{A_m}V^-(t+h,p^m,q^m)\right]-(2\tilde C+1)\varepsilon\\
 \geq &  \E_t\left[\Int_t^{t+h}H(s,X^{t,p,u}_s,Y_s^{t,q,\beta^0(u)})ds+V^-(t+h,X^{t,p,u}_{t+h},Y^{t,q,\beta^0(u)}_{t+h})\right]-(3\tilde C+1)\varepsilon\\
\geq& RHS-(3\tilde C+2)\varepsilon.
\end{align*}
The result follows as $\varepsilon$ can be chosen arbitrarily small.
\end{proof}

\section{Convexity properties}

A crucial step of our argumentation is to prove that $V^+$ and $V^-$ are convex in $p$ and concave in $q$. In opposition to previous works, where we were able to adapt without too much difficulty the splitting arguments of the repeated game setting, we cannot avoid here to have recourse to some technical machinery. We shall hide a part of this technicality in an appendix.

\pa The following lemma is a key argument for the splitting procedure. It relies on the predictable representation property for continuous martingales and shows how  a stochastic integral with respect to a Brownian motion of type (\ref{sdeXY}) can  mimic as close as possible the jump of a splitting martingale.

\begin{Lemma}
\label{martrep}
Let $p^1,p^2$ belong to the relative interior of $\Delta(I)$ and $\lambda_1,\lambda_2\in(0,1)$ with $\lambda_1+\lambda_2=1$. Set $p=\lambda_1p_1+\lambda_2p_2$.
For $h\in(0,T-t]$, let $Z$ be a $\sigma(B^1_r, r\in[t,t+h])$-measurable random variable such that $\P_t[Z=p^i]=\lambda_i,i \in\{ 1,2\}$. Then, for all $\varepsilon>0$, there exists $\bar u\in\UR^s(t)$ 
such that
\begin{equation}
\label{XZ}
 \E_t\left[|X^{t,p,\bar u}_{t+h}-Z|\right]\leq\varepsilon.
 \end{equation}
 \pa An analogue result holds for $q^1,q^2$ in the relative interior of $\Delta(J)$.
\end{Lemma}

\begin{proof} 
Since $\E_t[Z]=p$ and $Z$ is $\sigma(B^1_r, r\in[t,t+h])$-measurable, by the martingale representation theorem, there is some process $(a_s)_{s \in[t,t+h]}\in\UR(t)$ 
such that
\[ Z=p+\int_t^{t+h}a_r dB^1_r.\]
Since we have assumed that $Z$ belongs to $\Delta(I)$, Proposition \ref{Xu} applies:
\[ Z=X^{t,p,a}_{t+h}.\]
This process being a martingale such that $X^{t,p,a}_{t+h}\in\{ p_1,p_2\}$, it follows that, for all $s\in [t,t+h]$,  $X^{t,p,a}_s=\E_t[X^{t,p,a}_{t+h} | \FR^{T}_{t,s}]$ is almost surely a convex combination of $p_1$ and $p_2$, i.e. belongs to the line segment $[p_1,p_2]:=\{ \lambda p_1+(1-\lambda )p_2 \,|\, \lambda\in[0,1]\}$. Let $M=\R.(p_2-p_1)$, then for any vector $w\in M^{\perp}$, it holds that
\[ 0=\langle w,Z-p \rangle=\int_t^{t+h}  (w^t a_s) dB^1_s,\]
where $w^t$ denote the transpose of $w$. This implies that we have $ds\otimes \P_{t}$ almost surely $w^ta_s=0$. Considering a countable dense subset of vectors $w$ in $M^\perp$, we deduce that $a_s\in L$,  $ds\otimes \P_{t}$ almost surely, where $L=\{ A \in \R^{|I\times I|} \,|\, w^tA=0 , \; \forall w \in M^{\perp}\}$. Note that for all $A \in L$ and all $z\in \R^{|I|}$, we have $Az \in M$, since $\langle w,Az\rangle = w^tAz=0$ for all $w\in M^\perp$ and $M=(M^{\perp})^{\perp}$. 
 
\pa
Now it is well-known that there exists a sequence of simple processes $a^n \in \UR^s(t)$ 
 such that 
\begin{equation} \label{approx} 
\E_t\left[\sup_{s \in [t,t+h]} |\int_t^s \left(a_r-a^n_r\right) dB^1_r |^2 \right]\rightarrow 0,
\end{equation}
and we may choose this approximating sequence such that $a^n_s \in L$ $ds\otimes \P_{t}$ almost surely, since these approximations are constructed via averaging procedures (see e.g Lemma 2.4 p. 132 in \cite{KS}). It follows that for all $s\in [t,t+h]$, $\int_t^s a^n_s dB^1_s \in M$ almost surely.  
Let $\delta>0$ such that $p^1-\delta(p^2-p^1)$ and $p^2+\delta(p^2-p^1)$ belong to $\Delta(I)$.
Up to take a subsequence, we may assume that convergence in \eqref{approx} holds almost surely, and if we define 
\[ \tau_n=\inf \{ s \in [t,t+h] \,|\, p+\int_t^s  a^n_r dB^1_r \notin[p^1 - \delta(p^2-p^1), p^2+\delta (p^2-p^1)]\} \wedge (t+h),\]
then $\P_t[\t_n=t+h] \rightarrow 1$. Note that on the event $\{\tau_n=t+h\}$, we have 
\[ X^{t,p,a^n}_{t+h}=p+\int_t^{t+h}  a^n_r dB^1_r.\]
Therefore, if we define $\bar u=a^n$ for some $n$  such  that 
\[\P_t[\t_n<t+h]\leq \frac{\varepsilon}{4}, \quad \E_t\left[\sup_{s \in [t,t+h]} |\int_t^s \left(a_r-a^n_r\right) dB^1_r |^2\right]\leq \frac{\varepsilon^2}{4},\] 
we get
\[\begin{array}{rl}
\E_t\left[|X^{t,p,\bar  u}_{t+h}-Z|\ind_{\{\tau_n=t+h\}}\right]= &
\E_t\left[|\int_t^{t+h}(a^n_r-a_r)dB^1_r|\ind_{\{\tau_n=t+h\}}\right]\\
\leq &\E_t\left[\sup_{s\in[t,t+h]}|\int_t^s(a^n_r-a_r)dB^1_r|^2\right]^{1/2}\\
\leq &\frac\varepsilon 2.
\end{array}\]
and, since $|X^{t,p,\bar u}_{t+h}-Z|$ is bounded by 2,
\[ \E_t\left[|X^{t,p,\bar u}_{t+h}-Z|\ind_{\{\tau_n<t+h\}}\right]\leq \frac \varepsilon 2.\]
The result follows.
\end{proof}

\begin{Lemma}\label{lemma-orthogonal}
For any  process $a \in \UR(t)$ with $\E_t[\int_t^T |a_s|^2 ds]<\infty$, we have for all $t'\in[t,T]$,
\[ \E_t[ \int_t^{t'} a_s dB^1_s \,|\, (B^2_r)_{r \in [t,T]} ]=0.\]
\end{Lemma}

\begin{proof}
We need to show that, for all $i\in I$ and all $\sigma \{ B^2_r,r \in [t,T]\}$-measurable, bounded random variable $Z$,
$\E_t[Z\int_t^{t'} a^i_r dB^1_r]=0$,
where $(a^i_r)_{r\in[t,T]}$ denotes here the $i$-th line of the matrix-valued process $(a_r)_{r\in[t,T]}$.
Applying the martingale representation theorem to $Z$, we obtain:
\[ Z=z+\int_t^TH_rdB^2_r,\]
with $z=\E_t[Z]$ and $(H_r)_{r\in[t,T]}$ an $\R^{|J|}$-valued square-integrable process adapted to the augmentation of the filtration generated by $(B^2_r)_{r\in [t,T]}$.
Since $B^1$ and $B^2$ are independent, It\^{o} formula implies:
\[ \E_t\left[Z\int_t^{t'} a^i_r dB^1_r\right]=z\E_t\left[\int_t^{t'} a^i_r dB^1_r\right]+ \sum_{(i',j')\in I\times J}\E_t\left[\int_t^{t'}H_r^{j'}a^{i,i'}_r d\langle B^{1,i'},B^{2,j'}\rangle_r\right]=0,\]
where $(\langle B^{1,i'},B^{2,j'}\rangle_r)_{r \in[t,T]}$ denotes the quadratic covariation between $B^{1,i'}$ and $B^{2,j'}$.

\end{proof}

\begin{Proposition}\label{propconcave}
The functions $V^+,V^-$ are convex in $p$ and concave in $q$.
\end{Proposition}
\begin{proof}
We only prove that $V^-$ is convex in $p$ and concave in $q$, the proof for $V^+$ being similar.
\pa
1. $V^-$ is convex in $p$.
\pa
\noindent 1.1 Fix $p^1,p^2 \in \Delta(I)$ and $\lambda_1,\lambda_2\in(0,1)$ such that $\lambda_1+\lambda_2=1$, and set $p=\sum_i\lambda_ip^i$.
As $V^-$ is Lipschitz in $p$, it is sufficient to prove the inequality for $p^1,p^2$ in the relative interior of $\Delta(I)$. Let $0<\varepsilon<|p^1-p^2|$ and
consider an arbitrary strategy  $\beta\in\BR(t)$. Let $t=t_0<t_1<\ldots<t_m=T$  be the time grid associated to $\beta$. The control induced by $\beta$ on $[t,t_1)$ doesn't depend on the strategy of the opponent and is deterministic. Let us call it $(v^1_s)_{s\in[t,t_1)}$. \\
Fix $h\in (0,\varepsilon)$ such that $t+h\in (t,t_1)\cap\Q$, and
\[\E_t\left[|Y^{t,q,v^1}_{t+h}-q|\right]\leq \varepsilon.\]
\pa 1.2 Let the random variable $Z$ and $\bar u\in\UR^s(t)$ defined by Lemma \ref{martrep}. Set $u^1=\bar u|_{[t,t+h]}$.
With the help of $u^1$, we can define a continuation strategy for $\beta$: for all $\omega_1\in\Omega_{t,t+h}$, 
\[ (\omega_2,\u_2)\in \Omega_{t+h}\times U^s_{t+h} \rightarrow\beta_{\omega_1}(\omega_2,\u_2):=\beta(\omega_1\otimes \omega_2 , u^1(\omega_1)\otimes \u_2 )_{|[t+h,T]},\] 
where $u^1(\omega_1)\otimes \u_2(s)=u^1(\omega_1,s)$ if $s<t+h$ and $\u_2(s)$ else.
Remark that, for all $\omega_1\in\Omega_{t,t+h}$, $\beta_{\omega_1}$ is a well-defined strategy in $\BR(t+h)$.
\pa
Using the measurable selection argument developed in Theorem \ref{selmeas} of the appendix, we can find two
measurable process-valued maps $u^{2,1}$ and $u^{2,2}: \Omega_{t,t+h}\rightarrow \UR^s(t+h)$ such that with $\P_{t,t+h}$-probability larger than $1-\varepsilon$, the process $(u^{2,i}(\omega_1,\cdot))_s)_{s\in[t+h,T]}$ is $\varepsilon$-optimal for $\inf_{u\in\UR^s(t+h)}J(t+h,p^i,q,u,\beta_{\omega_1}(u))$ and has a  grid  which is independent of $\omega_1$.
\pa
This allows us to define a new control in $\UR^s(t)$:
\[
\tilde u_s(\omega)=\tilde u_s(\omega_1\otimes\omega_2)=\left\{
\begin{array}{ll}
u^1_s(\omega_1),&\mbox{ if }s\in[t,t+h),\\
u^2_s(\omega_1,\omega_2),& \mbox{ if }s\in[t+h,T],
\end{array}\right.\]
with
\[
u^2_s(\omega_1,\cdot)=\left\{ 
\begin{array}{rl}
u^{2,1}_s(\omega_1,\cdot),&\mbox{ if }s\in[t+h,T],  Z(\omega_1)=p^1,\\
u^{2,2}_s(\omega_1,\cdot),&\mbox{ if }s\in[t+h,T],  Z(\omega_1)=p^2.
\end{array}
\right.\]

\pa 1.3 By the same arguments as in the proof of Proposition \ref{propdpp}, we can write for all $(\omega_1,\omega_2)$:
\[  \forall s \in [t+h,T], \; X^{t,p,\tilde u}_s(\omega_1\otimes \omega_2)= X^{t+h,X(\omega_1),u^2(\omega_1)}_s(\omega_2), \; \text{with}\;  X(\omega_1):= X^{t,p,u^1}_{t+h}(\omega).\]
and
\[ \forall s \in [t+h,T], \; Y^{t,q,\beta(\tilde u)}_s(\omega_1\otimes \omega_2)= Y^{t+h,Y(\omega_1),\beta_{\omega_1}(u^2(\omega_1))}_s(\omega_2), \; \text{with}\;  Y(\omega_1):=Y^{t,q,v^1}_{t+h}(\omega).\]
We get
\begin{equation}
\label{AB}
\begin{array}{rl}
J(t,p,\tilde u,\beta(\tilde u))=&
\E_t\left[\int_t^{t+h}H(s,X^{t,p,\tilde u}_s,Y^{t,q,\beta(\tilde u)}_s)ds\right]\\
&+\int_{\Omega_{t,t+h}}\E_{t+h}\left[\int_{t+h}^TH(s,X^{t+h,X(\omega_1),u^2(\omega_1)}_s,Y_s^{t+h,Y(\omega_1),\beta_{\omega_1}(u^2(\omega_1))})ds\right]d\P_{t,t+h}(\omega_1)\\
:=& A+B.
\end{array}\end{equation}
Obviously $|A|\leq Ch\leq C\varepsilon$.\\
Concerning the second term of the right hand side, set
\[ B':= \int_{\Omega_{t,t+h}}\E_{t+h}\left[\int_{t+h}^TH\left(s,X^{t+h,Z(\omega_1),u^2(\omega_1)}_s,Y^{t+h,q,\beta_{\omega_1}(u^2(\omega_1))}_s\right)ds\right]\P_{t,t+h}(\omega_1).\]
Choosing $\varepsilon$ small enough, this expression can be arbitrarily close to $B$. Indeed:
\begin{equation}
\label{BB'} \begin{array}{rl}
|B-B'|\leq & C\int_{\Omega_{t,t+h}}\E_{t+h}\big[\int_{t+h}^T\big(|X^{t+h,Z(\omega_1),u^2(\omega_1)}-X^{t+h,X(\omega_1),u^2(\omega_1)}|\\
&\qquad +
|Y^{t+h,q,\beta_{\omega_1}(u^2(\omega_1))}-Y^{t+h,Y(\omega_1),\beta_{\omega_1}(u^2(\omega_1))}|\big)ds\big]d\P_{t,t+h}(\omega_1)\\
\leq& C\bar C T \int_{\Omega_{t,t+h}}\left(|Z(\omega_1)-X(\omega_1)|+|q-Y(\omega_1)|\right) d\P_{t,t+h}(\omega_1)\\
=&C\bar C T\E_{t,t+h}\left[ |Z-X^{t,p,\bar u}_{t+h}|+|q-Y^{t,q,v^1}_{t+h}|\right]\\
\leq & 2C\bar C T\varepsilon.
\end{array}
\end{equation}
Then we can estimate $B'$ : 
\[\begin{array}{rl}
B'
=&\sum_{i=1,2}\int_{\{Z(\omega_1)=p^i\}}\E_{t+h}\left[\int_{t+h}^TH(s,X^{t+h,p^i,u^{2}(\omega_1)}_s,Y^{t+h,q,\beta_{\omega_1}(u^2(\omega_1))}_s)ds\right]d\P_{t,t+h}(\omega_1)\\
=&\sum_{i=1,2}\int_{\{Z(\omega_1)=p^i\}}J(t+h,p^i,q,u^2,\beta_{\omega_1}(u^2))d\P_{t,t+h}(\omega_1)\\
= & \sum_{i=1,2}\int_{\{Z(\omega_1)=p^i\}}J(t+h,p^i,q,u^{2,i},\beta_{\omega_1}(u^{2,i}))d\P_{t,t+h}(\omega_1).
\end{array}\]
Because of the $\varepsilon$-optimality of $u^{2,i}$ with probability $(1-\varepsilon)$ and of the fact that $V^-$ is Lipschitz in $t$, it follows that
\begin{equation}
\label{B'}
\begin{array}{rl}
B'\leq & \sum_{i=1,2}\int_{\{Z(\omega_1)=p^i\}}\sup_{\beta'\in\BR(t+h)}\inf_{u\in\UR^s(t+h)}
J(t+h,p^i,q,u,\beta'(u)) dP_{t,t+h}(\omega_1)+(4CT+1)\varepsilon
\\
= & \sum_{i=1,2}\lambda_iV^-(t+h,p^i,q)+(4CT+1)\varepsilon\\
\leq & \sum_{i=1,2}\lambda_iV^-(t,p^i,q)+(4CT+8C+1)\varepsilon.
\end{array}
\end{equation}
Summing up (\ref{AB}), (\ref{BB'}) and (\ref{B'}), we get
\[ J(t,p,\tilde u,\beta(\tilde u))\leq \lambda_1V^-(t,p^1,q)+\lambda_2V^-(t,p^2,q)+K\varepsilon,\]
where the constant $K$ depends only on the parameters of the game. The result follows by a standard argument.

\pa
2. $V^-$ is concave with respect to $q$. 
 \pa
We fix again $\lambda_1,\lambda_2\in (0,1)$ such that $\lambda_1+\lambda_2=1$ and choose now
 $q^1,q^2,q\in\Delta(J)$ with $q=\lambda_1 q^1+\lambda_2q^2$. 
As in step 1, we can take $q^1,q^2$ in the relative interior of $\Delta(J)$.
Fix $h>0$ such that $t+h\in(t,T]\cap\Q$. 
Using Lemma \ref{martrep}, we can find $\bar v\in\VR^s(t)$ and $Z$ a $\sigma(B^2_r, r\in[t,t+h])$-measurable, $\Delta(J)$-valued random variable such that $\P_t[Z=q^i]=\lambda_i, i=1,2$ and $\E_t[|Y^{t,q,\bar v}_{t+h}-Z|]\leq h$.
\pa
Let us fix a measurable partition $(O_r)_{r\in\{ 1,\ldots,R\}}$ of $\Delta(I)$ of mesh $h$ and $(p^r)_{r\in\{ 1,\ldots,R\}}\subset\Delta(I)$ with $p^r\in O_r$ for all $r\in\{ 1,\ldots,R\}$.
For $r\in\{ 1,\ldots,R\}$ and $i\in\{ 1,2\}$,  let $\beta^{r,i}\in\BR(t+h)$ be $h$-optimal for $V^-(t+h,p^r,q^i)$ and $A^u_{r,i}=\{ X^{t,p,u}_{t+h}\in O_r, Z=q^i\}$. Remark that the sets $A^u_{r,i}, i\in\{ 1,2\}, r\in\{1,\ldots,R\}$ form a partition of $\Omega_{t,t+h}$. Then, as in the proof of Proposition \ref{propdpp}, we  can prove that the following map from $\Omega_t\times U^s_t$ to $V^s_t$ defines an admissible strategy for Player 2: 
\[ \beta(\omega,u)(s)=
\left\{ 
\begin{array}{ll}
\bar v_s,& \mbox{ if }s\in[t,t+h),\\
\beta^{r,i}(\omega_2,u|_{[t+h,T]})(s),&\mbox{ if } (s,\omega_1)\in [t+h,T]\times A^u_{r,i}
\end{array}
\right. .\]

Now, for all $u\in\UR^s(t)$, by the definition of $\bar v$ and the Lipschitz continuity of $H$,
\[\begin{array}{rl}
J(t,p,q,u,\beta)& \geq \\
\sum_{r,i}\int_{A^u_{r,i}} &\E_{t+h}\left[
\int_{t+h}^TH(s,X^{t+h,p^r,u|_{[t+h,T]}(\omega_1)}_s,Y^{t+h,q^i,\beta^{r,i}(u|_{[t+h,T]}(\omega_1))})ds\right]d\P_{t,t+h}(\omega_1)-C(2T+1)h\\
\geq &\sum_{r,i}\E_t\left[V^-(t+h,p^r,q^i)\ind_{A^u_{r,i}}\right]-(C(2T+1)+1)h\\
&\mbox{because of the $h$-optimality of $\beta^{r,i}$,}\\
\geq &\sum_i\E_t\left[V^-(t,X^{t,p,u}_{t+h},q^i) \ind_{Z=q^i}\right]-\tilde C h
\end{array}\]
where in the last inequality we have used the Lipschitz continuity of $V^-$ in $(t,p)$, with $\tilde C=C(9+2T)+C\bar CT+1$.
To conclude, note that by Lemma \ref{lemma-orthogonal}, we have 
\[ \E_t[X^{t,p,u}_{t+h} | (B^2_r)_{r \in [t,T]}]=p .\]
Since $Z$ is $\sigma(B^2_r, r\in[t,t+h])$-measurable, and using Jensen's inequality for conditional expectations (which we may apply as we already proved that $V^-$ is convex in $p$), we obtain that for $i=1,2$:
\[ \E_t\left[V^-(t,X^{t,p,u}_{t+h},q^i) \ind_{Z=q^i}| (B^2_r)_{r \in [t,T]} \right]= \ind_{Z=q^i}\E_t\left[V^-(t,X^{t,p,u}_{t+h},q^i) | (B^2_r)_{r \in [t,T]} \right]\geq  \ind_{Z=q^i} V^-(t,p,q^i).\]
We deduce that
\[ J(t,p,q,u,\beta) \geq \sum_i \lambda_i V^{-}(t,p,q^i) -\tilde C h,\]
and since this relation holds true for for all $u\in\UR^s(t)$, it implies that 
\[ V^-(t,p,q) \geq \sum_i \lambda_i V^{-}(t,p,q^i) -\tilde C h.\]
The result follows since $h$ can be chosen arbitrarily small.
\end{proof}

\section{Viscosity solution} \label{sectionviscosity}

\pa Now we have all the ingredients to establish the main result. First, let us recall the definition of sub- and supersolutions for (\ref{hji}) given in \cite{c2}.

\begin{Definition}
1.
A function $w:[0,T]\times\Delta(I)\times\Delta(J)\mapsto\R$ is called a supersolution of (\ref{hji}) if it is l.s.c. and satisfies, for any smooth test function $\varphi:[0,T]\times\Delta(I)\times\Delta(J)\mapsto\R$ and $(t,p,q)\in\Delta(I)\times\mbox{Int}(\Delta(J))$: if $\varphi-w$ has a local maximum at $(t,p,q)$ then
\[ \max\left\{\min\left\{ \frac{\partial \varphi}{\partial t}(t,p,q)+H(t,p,q);\lambda_{min}(p,D^2_p\varphi(t,p,q))\right\};\lambda_{max}(q,D^2_q\varphi(t,p,q))\right\}\leq 0 .\]
\pa 2.
The function $w:[0,T]\times\Delta(I)\times\Delta(J)\mapsto\R$ is called a subsolution of (\ref{hji}) if it is u.s.c. and satisfies, for any smooth test function $\varphi:[0,T]\times\Delta(I)\times\Delta(J)\mapsto\R$ and $(t,p,q)\in\mbox{Int}(\Delta(I))\times\Delta(J)$: if $\varphi-w$ has a local minimum at $(t,p,q)$ then
\[ \min\left\{\max \left\{  \frac{\partial \varphi}{\partial t}(t,p,q)+H(t,p,q);\lambda_{max}(q,D^2_q\varphi(t,p,q))\right\};\lambda_{min}(p,D^2_p\varphi(t,p,q))\right\}\geq 0 ;\]
\pa 3. $w$ is called a solution of (\ref{hji}), if it is both a super- and a subsolution.
\end{Definition}

\begin{Theorem}
The game has a value $V=V^-=V^+$, which is the unique Lipschitz viscosity solution of (\ref{hji}).
\end{Theorem}

\begin{proof}
We know already that $V^-,V^+$ are convex in $p$, concave in $q$ and Lipschitz in all their variables. Moreover, the uniqueness of the solution of (\ref{hji}) in this class of functions has already been established in \cite{cr1}. 

\pa Let us prove that $V^-$ is a supersolution:\\
Let $\varphi$ be a map from $[0,T]\times\Delta(I)\times\Delta(J)$ to $\R$ which is  $C^2$ in all its variables and such that, for some fixed $(t,p,q)\in [0,T]\times\iNt(\Delta(I))\times\Delta(J)$, $\varphi-V^-$ has a local maximum at $(t,p,q)$. Without loss of generality we can suppose that $\varphi(t,p,q)=V^-(t,p,q)$. Since $V^-$ is concave in $q$, we have $\lambda_{max}(q,D^2_q\varphi(t,p,q))\leq 0$. Suppose that $\lambda_{min}(p,D^2_p\varphi(t,p,q))\leq 0$. Then the relation
\[ \max\left\{\min\left\{ \frac{\partial \varphi}{\partial t}(t,p,q)+H(t,p,q);\lambda_{min}(p,D^2_p\varphi(t,p,q))\right\};\lambda_{max}(q,D^2_q\varphi(t,p,q))\right\}\leq 0 ;\]
is satisfied. Suppose now that $\lambda_{min}(p,D^2_p\varphi(t,p,q))> 0$. In this case we have to prove that $\frac{\partial \varphi}{\partial t}(t,p,q)+H(t,p,q)\leq 0$. But, as in \cite{c2} (see claim 3.3 in the proof of Theorem 3.3), we can find $\bar{h}\in(0,T-t)$ and $\delta>0$ such that, for all $p'\in\Delta(I)$ and $s\in[t,t+\bar{h}]$,
\begin{equation}
\label{conv}
 V^-(s,p',q)\geq \varphi(s,p,q)+D_p\varphi(s,p,q)\cdot(p'-p)+\delta|p'-p|^2.
\end{equation}
By the dynamic programming principle \ref{propdpp}, we have for $h>0$, with $t+h\in[t,t+\bar{h}] \cap \Q$
\[\begin{array}{rl}
0\geq &\sup_{\beta\in\BR(t)}\inf_{u\in\UR^s(t)} \E_t\left[ \int_t^{t+h}H(s,X^{t,p,u}_s,Y^{t,q,\beta(u)}_s)ds+V^-(t+h,X^{t,p,u}_{t+h},Y^{t,q,\beta(u)}_{t+h})-V^-(t,p,q)\right]\\
\geq &\inf_{u\in \UR^s(t)}\E_t\left[ \int_t^{t+h}H(s,X^{t,p,u}_s,q)ds+V^-(t+h,X^{t,p,u}_{t+h},q)-V^-(t,p,q)\right],
\end{array}\]
where the second inequality is obtained by choosing the particular strategy $\beta \equiv 0$. 
\pa
Now we take $\bar u \in \UR^s(t)$ which is $h^2$-optimal for this last right hand side term, and remark that $\E_t[D_p\varphi(t+h,p,q)\cdot(X^{t,p,\bar u}_{t+h}-p)]=0$. Then, from relation (\ref{conv}),
\[
 \E_t\left[ \int_t^{t+h}H(s,X^{t,p,\bar u}_s,q)ds+\varphi(t+h,p,q)-\varphi(t,p,q)+\delta|X^{t,p,\bar u}_{t+h}-p|^2\right]\leq h^2.
\]
This implies that 
\begin{equation}
\label{nino}
\E_t\left[ \int_t^{t+h}(H(s,X^{t,p,\bar u}_s,q)+\frac{\partial \varphi}{\partial t}(s,p,q))ds\right]\leq h^2.
\end{equation}
and also (since $H$ and $\frac{\partial \varphi}{\partial t}$ are bounded) that there exists some constant $C'$ such that
$\E_t\left[|X^{t,p,\bar u}_{t+h}-p|^2\right]\leq C' h$. Therefore, by Jensen's inequality, we have for all $s\in[t,t+h]$,
\[ \E_t\left[|X^{t,p,\bar u}_s-p|^2\right]\leq  C'h.\]
Using this relation in (\ref{nino}), we get
\[ \begin{array}{rl}
\int_t^{t+h}(H(s,p,q)+\frac{\partial \varphi}{\partial t}(s,p,q))ds\leq &\E_t\left[ \int_t^{t+h}(H(s,X^{t,p,\bar u}_s,q)+\frac{\partial \varphi}{\partial t}(s,p,q))ds\right]\\
&\qquad\qquad\qquad+
C\int_t^{t+h}(\E_t[|X^{t,p,\bar u}_s-p|^2])^{1/2}ds\\
\leq &h^2 +C\sqrt{C'}h^{3/2}.
\end{array}\]
The result follows.\\
\pa The remaining of the proof is standard: 
By symmetric arguments, $V^+$ is a subsolution of (\ref{hji}). From the comparison theorem of \cite{cr1} we deduce that $V^+\leq V^-$. But we know already that $V^-\leq V^+$. It follows that the two value functions are equal and both viscosity sub- and supersolution of \ref{hji}.
\end{proof}

\section{Appendix}

Let us start with easy remarks on the set of  simple trajectories  and strategies.
\pa
Recall that $U_{t,t'}$ denote the set of equivalence classes (with respect to the Lebesgue measure) of Borel measurable maps from $[t,t']$ to $U$.  Let $\Pi_{t,t'}$ denote the countable set of all finite partitions $\pi=\{t=t_0<t_1<...<t_m=t'\}$ with $t_1,...,t_{m-1}$ in $(t,t')\cap \Q$. Let $|\pi|=m$ denote the size of the partition (number of intervals). Given $\pi \in \Pi_{t,t'}$, let  $U^\pi_{t,t'}$ denote the subset of maps that are piecewise constant on $\pi$. 
Remark that $U^\pi_{t,t'}$ is a Borel subset of $U_{t,t'}$. Indeed, one may define $U^\pi_{t,t'}$ by a countable number of measurable constraints as follows:\\
\begin{multline} \forall j  \in \{0,...,m-1\},  \forall s,s',c,c' \in [t_j,t_{j+1}] \cap \Q \; \text{with}\; s<s' \; \text{and} \; c<c', \;\\
 \frac{1}{s'-s}\int_{s}^{s'} \u(r)dr=\frac{1}{c'-c}\int_{c}^{c'} \u(r)dr.
 \end{multline}
The set of simple trajectories $U^s_{t,t'}=\cup_{\pi \in \Pi_{t,t'}} U^\pi_{t,t'}$ is a Borel subset of $U_{t,t'}$ as a countable union of Borel subsets. The measurability of a map defined on $U^s_{t,t'}$ is therefore equivalent to the measurability of its restriction to $U^\pi_{t,t'}$ for each $\pi$.
Note finally that $U^\pi_t$ is homeomorphic to $U^{m}$ by identifying $\u$ with $\left(\frac{1}{t_{j+1}-t_j} \int_{t_j}^{t_{j+1}} \u(s) ds\right)_{j=0,...,m-1}$. Therefore, in order to prove that some map $g: \Omega_{t,t'}\times U^s_{t,t'}  \rightarrow V$ is Borel-measurable, it is sufficient to prove that for each $\pi$, there exists a measurable map $g^\pi: \Omega_{t,t'} \times U^{m}  \rightarrow V$ which coincides with the restriction of $g$ to $U^\pi_{t,t'}$ in the sense that
\[ \forall \omega \in \Omega_t, \forall \u \in U^\pi_{t,t'}, \; g(\omega,\u)= g^{\pi}\left(\omega,\left(\frac{1}{t_{j+1}-t_j} \int_{t_j}^{t_{j+1}} \u(s) ds\right)_{j=0,...,m-1} \right).\]
We will use this remark in the following Lemma.

\begin{Lemma}\label{SDE-borel}
Given $t,t' \in [0,T]$ with $t<t'$, 
there exist Borel maps: 
\[(p,\omega,\u) \in \Delta(I)  \times \Omega_{t,t'} \times U^s_{t,t'} \rightarrow  \Phi^X_{t,t'} (p,\omega,\u) \in C([t,t'],\R^{|I|}),\]
\[(q,\omega,\v) \in \Delta(J)  \times \Omega_{t,t'} \times V^s_{t,t'} \rightarrow  \Phi^Y_{t,t'} (q,\omega,\v) \in C([t,t'],\R^{|J|}),\]
such that given any $u\in \UR^s(t)$ and $\beta \in \BR(t)$, 
then with $\P_t$-probability $1$, we have
\[ (X^{t,p,u}(\omega),Y^{t,q,\beta(u)}(\omega))=(\Phi^X_{t,T}(p,\omega,u(\omega)),\Phi^Y_{t,T}(q,\omega,\beta(u)(\omega)).\]
Moreover, if $h \in(0,T-t]$, then for all $s \in [t,T]$,
\[ \Phi^X_{t,T}(p,\omega,\u)(s)= \left\{ \begin{matrix} \Phi^X_{t,t+h}(p,\omega|_{[t,t+h]},\u|_{[t,t+h]}) & \text{if $s\leq t+h$,} \\
 \Phi^X_{t+h,T}(\Phi^X_{t,T}(p,\omega,\u)(t+h),\omega|_{[t+h,T]} - \omega(t+h),\u|_{[t+h,T]})(s) & \text{ if $s> t+h$}.
\end{matrix} \right. \] 
It follows that, using the notations of the proof of Proposition \ref{propdpp}, we have almost surely
\[  \forall s \in [t+h,T], \; X^{t,p,u}_s(\omega_1\otimes \omega_2)= X^{t+h,X(\omega_1),u(\omega_1)}_s(\omega_2),\]
where $X(\omega_1):=\Phi^X_{t,t+h}(p,\omega_1,u(\omega_1))$ is $\P_t$-almost surely equal to $X^{t,p,u}_{t+h}(\omega)$ .
\pa 
Similar results hold for $\Phi^Y$.
\end{Lemma}

\begin{proof}
The proof follows from the fact that for simple trajectories of the controls, the stochastic integrals appearing in the proof of Proposition \ref{prop1} is explicit and pathwise. We only prove the Lemma for $\Phi^X$ as the construction is similar for $\Phi^Y$.  
\pa
Let $(p,\omega,\u) \in \Delta(I)  \times \Omega_{t,t'} \times U^s_{t,t'}$, and define $X^1=X^1(p,\omega,\u)\in C([t,t'],\R^{|I|})$ by: 
\[ (p,\omega,\u) \rightarrow X^1=p+\int_t^\cdot T_p \u(s)d B^1_s \in C([t,t'],\R^{|I|}),\]
using the standard pathwise definition for the stochastic integral of simple processes.  
\pa
Then, for a generic trajectory $x$ in $C([t,t'],\R^{|I|})$ and for $k \in \{1,...,|I|-1\}$, let 
\[  \tau^k(x)= \inf \{ s \in [t,t'] \,|\, | \{ i \in I | x^i(s) \neq 0 \}| \leq |I|-k\}.\]
Note that $\tau^k$ being the hitting time of a closed set, it is a stopping time of the raw filtration of $C([t,t'],\R^{|I|})$, and thus a Borel map from $C([t,t'],\R^{|I|})$ to $[t,T] \cup\{+\infty\}$. 
\pa
Then, define by induction for $k=1,...,|I|-1$, $X^{k+1}(p,\omega,\u)$ by 
\[ X^{k+1}_s= X^k_{s\wedge\tau^k(X^k)}  +\int_{\tau^k(X^k)}^{s\vee \tau^k(X^k)} P_{X^k_{\tau^k(X^k)}}\u(r) dB^1_r,\;  s \in[t,t'].\]
We set $\Phi^X_{t,t'}(p,\omega,\u):= X^{|I|}$ for $\u \in U^s_{t,t'}$. 
In order to prove that $\Phi^X_{t,t'}$ is Borel, it is sufficient to prove that for each partition $\pi=\{ t=t_0<t_1<...<t_m=t'\}$ in $\Pi_{t,t'}$, the restriction of $\Phi^X_{t,t'}$ to $\Delta(I)  \times \Omega_{t,t'} \times U^\pi_{t,t'}$ is Borel.
Let $\pi \in\Pi_{t,t'}$. In the following, given some vector $(\u_j)_{j=0,...,m-1} \in U^m$, let $\u$ denote the trajectory $\u(s)= \sum_{j=0}^{m-1} \u_j \ind_{s \in [t_j,t_{j+1})} \in U^\pi_{t,t'}$. We have  to show that
\[  \begin{matrix} \Delta(I)\times\Omega_{t,t'}\times U^m & \rightarrow & C([t,t'],\R^{|I|}) \\ (p,\omega,(\u_j)_{j=0,...,m-1})& \mapsto& X^{|I|} \end{matrix},\]
is Borel. Note at first that the map
\[ (p,\omega,(\u_j)_{j=0,...,m-1}) \rightarrow X^1=p+\int_t^\cdot T_p \u(s)d B^1_s(\omega) \in C([t,t'],\R^{|I|}),\]
is Borel, using the standard pathwise definition for the stochastic integral of simple processes. It follows that $\tau^1(X^1)$ is also Borel by composition. 
Remark then that the map $(p,\omega,(\u_j)_{j=0,...,m-1})\mapsto X^2$ is the composition of the Borel measurable maps:
\begin{itemize}
\item  $ \begin{matrix} \Delta(I)\times\Omega_{t,t'}\times U^m & \rightarrow & \Omega_{t,t'}\times U^m \times C([t,t'],\R^{|I|})\times ([t,t']\cup\{+\infty\}) \\ 
(p,\omega,(\u_j)_{j=0,...,m-1}) & \mapsto & (\omega,(\u_j)_{j=0,...,m-1},X^1,\tau^1(X^1))
\end{matrix}$

\item $\begin{matrix}  \Omega_{t,t'}\times U^m \times C([t,t'],\R^{|I|})\times ([t,t']\cup\{+\infty\}) & \rightarrow & C([t,t'],\R^{|I|}) \\
(x,\omega,(\u_j)_{j=0,...,m-1},\delta) & \mapsto &\left(x(s \wedge \delta)+\ind_{\delta \leq s}\int_{s}^{\delta}P_{x(\delta)} \u(r)dB^1_r(\omega) \right)_{s\in[t,t']}
\end{matrix}$
\end{itemize}
The measurability of $(p,\omega,\u)\mapsto X^k(\omega)$ for $k\geq 3$ and finally from $\Phi^X_{t,t'}$ follows by induction.

\pa
To conclude the first point of the lemma, note that the construction coincides with the one described in Proposition \ref{prop1} for simple processes with $t'=T$, so that for all $u\in \UR^s(t)$, we have with $\P_t$-probability $1$:
\[ X^{t,p,u}(\omega)=\Phi^X_{t,T}(p,\omega,u(\omega)).\]
Since our definition of simple processes uses the raw filtration, it follows that  $\Phi^X_{t,T}(p,\omega,u(\omega))$ is a solution to \eqref{sdeX} which is adapted to the raw filtration of $B$. 
\pa
Let us prove the second part of the lemma. Let $\u\in U^{\pi}_{t}$ and for $\omega \in \Omega_{t}$, let us denote:
\[\omega_1=\omega|_{[t,t+h]} \in \Omega_{t,t+h},\; \omega_2=\omega|_{[t+h,T]}-\omega(t+h)\in \Omega_{t+h}.\]
\[ \u_1=\u|_{[t,t+h]} \in U^{s}_{t,t+h}, \;\u_2=\u|_{[t+h,T]} \in U^{s}_{t+h}.\]
At first, let $s\in [t,t+h]$, and recall that $\Phi^X_{t,T}(p,\omega,\u)(s)=X^{|I|}_s$ where  
\[ X^1_s = p+\int_t^s T_p \u(r)d B^1_r(\omega)\]
and for $k\geq 2$
\[ X^{k+1}_s= X^k_{s\wedge\tau^k(X^k)}  +\int_{\tau^k(X^k)}^{s\vee \tau^k(X^k)} P_{X^k_{\tau^k(X^k)}}\u(r) dB^1_r(\omega),\]
with $\tau^k(x)= \inf \{ s \in [t,T] \,|\, | \{ i \in I | x^i(s) \neq 0 \}| \leq |I|-k\}$.
Similarly, $\Phi^X_{t,t+h}(p,\omega_1,\u_1)=\tilde{X}^{|I|}_s$ where  
\[ \tilde{X}^1_s = p+\int_t^s T_p \u_1(r)d B^1_r(\omega_1)\]
and for $k\geq 2$
\[ \tilde{X}^{k+1}_s= \tilde{X}^k_{s\wedge \tilde{\tau}^k(\tilde{X}^k)}  +\int_{\tilde{\tau}^k(\tilde{X}^k)}^{s\vee \tilde{\tau}^k(\tilde{X}^k)} P_{\tilde{X}^k_{\tilde{\tau}^k(\tilde{X}^k)}}\u_1(r) dB^1_r(\omega_1),\]
with $\tilde{\tau}^k(x)= \inf \{ s \in [t,t+h] \,|\, | \{ i \in I | x^i(s) \neq 0 \}| \leq |I|-k\}$. By construction $X^1=\tilde{X}^1$ on $[t,t+h]$, and therefore $\tilde{\tau}^1(\tilde{X}^1)=\tau^1(X^1)$ if $\tau^1(X^1)\leq t+h$ and $\tilde{\tau}^1(\tilde{X}^1)=+\infty$ otherwise. It implies in turn that $X^2=\tilde{X}^2$ on $[t,t+h]$ and therefore that $\tilde{\tau}^2(\tilde{X}^2)=\tau^2(X^2)$ if $\tau^2(X^2)\leq t+h$ and $\tilde{\tau}^2(\tilde{X}^2)=+\infty$ otherwise. The equality ${X}^{|I|}_s=\tilde{X}^{|I|}_s$ follows by induction on $k$. 
\pa
Let us now assume that $s\in (t+h,T]$. Let $\hat{p}=X^{|I|}_{t+h}=\Phi_{t,T}^X(p,\omega,\u)(t+h)$ and let $k^*=|I|-|S(\hat{p})|$ (recall that $S(x)=\{ i\in I, x^i\neq 0\}$). Using the convention $\tau^0=t$ and $X^0=p$, we have by construction $\tau^{k^*}(X^{k^*})\leq t+h$ and $P_{\hat{p}}=P_{X^{k^*}_{\tau^{k^*}(X^{k^*})}}$. It follows that
\[ X^{k^*+1}_s= \hat{p}+ \int_{t+h}^s P_{\hat{p}}\u(r)dB^1_r(\omega).\]
On the other hand, we have $\Phi^X_{t+h,T}(\hat{p},\omega_2,\u_2)(s)=\hat{X}^{|I|}_s$ where 
\[ \hat{X}^1_s = \hat{p}+\int_t^s P_{\hat p} \u_2(r)d B^1_r(\omega_2)\]
and for $k\geq 2$
\[ \hat{X}^{k+1}_s= \hat{X}^k_{s\wedge \hat{\tau}^k(\hat{X}^k)}  +\int_{\hat{\tau}^k(\hat{X}^k)}^{s\vee \hat{\tau}^k(\hat{X}^k)} P_{\hat{X}^k_{\hat{\tau}^k(\hat{X}^k)}}\u_2(r) dB^1_r(\omega_2),\]
where $\hat{\tau}^k(x)= \inf \{ s \in [t+h,T] \,|\, | \{ i \in I | x^i(s) \neq 0 \}| \leq |I|-k\}$. By induction, one checks easily that for all $k \leq k^*$, we have $\hat{X}^k_{t+h}=\hat{p}$ and $\hat{\tau}^k(\hat{X}^k))=t+h$ so that 
\[  \hat{X}^{k^*+1}_s= \hat{p} +\int_{t+h}^{s} P_{\hat{p}}\u_2(r) dB^1_r(\omega_2).\]
We deduce that $X^{k^*+1}=\hat{X}^{k^*+1}$ on $[t+h,T]$ and therefore $\hat{\tau}^{k^*+1}(\hat{X}^{k^*+1})=\tau^{k^*+1}(X^{k^*+1})$. It implies in turn that $X^{k^*+2}=\hat{X}^{k^*+2}$ on $[t+h,T]$ and therefore $\hat{\tau}^{k^*+2}(\hat{X}^{k^*+2})=\tau^{k^*+2}(X^{k^*+2})$. The equality ${X}^{|I|}_s=\hat{X}^{|I|}_s$ follows by induction.

\end{proof}

\begin{Definition}
Let $\;\UR^{c}(t)$ denote the set of simple controls that are continuous with respect to $\omega$, meaning that there exist $t=t_0<...<t_m=T$ with $t_1,...,t_{m-1}\in \Q$  and continuous maps $g_j$ from $\Omega_{t,t_j}$ to $U$ such that
\begin{equation}
\label{u}
\forall s\in [t,T), \; u(\omega, s)=\sum_{j=0}^{m-1}\ind_{[t_j,t_{j+1})}(s)g_j(\omega|_{[t,t_j]}).
\end{equation}
\end{Definition}


%

\begin{Lemma}
\label{Vc}
Let $p \in \Delta(I)$, $u \in \UR^s(t)$ and $\beta \in \BR(t)$. Then, for all $\varepsilon$ there exists $\tilde{u} \in \UR^c(t)$ such that 
\[ \E_t\left[\sup_{s \in [t,T]}\left(|X^{t,p,u}_s-X^{t,p,\tilde{u}}_s|^2+|Y^{t,q,\beta(u)}_s-Y^{t,q,\beta(\tilde{u})}_s|^2\right) \right] \leq \varepsilon.\]
\end{Lemma}
\begin{proof}
For any $u \in \UR^s(t)$, there exist $t=t_0<....<t_m=T$ and measurable maps $g_j$ from $\Omega_{t,t_j}$ to $U$ as defined in (\ref{u}).
Using inner regularity of measures and Lusin's theorem, for each $g_j$, there exists a compact set $K_j\subset \Omega_{t,t_j}$ such that $\P_{t,t_j}(K_j) \geq 1- \frac{\varepsilon}{8m}$ and such that the restriction of $g_j$ to $K_j$ is continuous. Let us consider some continuous map $f_j$ on $\Omega_{t,t_j}$ which coincides with $g_j$ on $K_j$ (Tietze extension theorem). Define $\tilde{u}$ by 
\[\tilde{u}(\omega, s)=\sum_{j=0}^{m-1}\ind_{[t_j,t_{j+1})}(s)f_j(\omega|_{[t,t_j]}).\]
The conclusion follows easily since both processes are bounded and coincide on the set 
\[ K=\{ \omega \in \Omega_t \,|\, \forall j, \; \omega|_{[t,t_j]} \in K_j \},\]
and $\P_t(K)\geq 1-\frac{\varepsilon}{8}$. 
\end{proof}

\pa
Let us now state the measurable selection result used in the proof of Proposition \ref{propconcave}. 
\begin{Theorem}
\label{selmeas}
Let $\varepsilon>0$, $(x,y) \in \Delta(I)\times \Delta(J)$, a process $u^1 \in \UR^s(t)$ and a strategy $\beta \in \BR(t)$.
Let $h>0$ such that $t+h \in (t,T)\cap \Q$ and smaller than the first point in the grid of the strategy $\beta$ so that $v:=\beta(u^1)|_{[t,t+h]}$ is a deterministic process independent of $u^1$. Recall the identification $\Omega_t=\Omega_{t,t+h}\times\Omega_{t+h}$ with $\omega=\omega_1\otimes\omega_2$, where $\omega_1:=\omega|_{[t,t+h]}$ and $\omega_2:=\omega|_{[t+h,T]}-\omega(t+h)$ and define the continuation strategy $\beta_{\omega_1}$ as in the proof of  Proposition \ref{propconcave}.
\pa
There exists a control $u \in \UR^s(t)$, which is equal to $u^1$ on $[t,t+h)$, and such that for all $\omega_1$ in a set of probability at least $1-\varepsilon$, $u^2(\omega_1)(\cdot):=u|_{[t+h,T]}(\omega_1\otimes \cdot)$ is an $\varepsilon$-best reply to the continuation strategy $\beta_{\omega_1}$ in the game starting at time $t+h$ with initial conditions $(x,y)$.
\end{Theorem}

\begin{proof}
Given $\pi\in \Pi_{t+h}$, let $\UR^\pi(t+h)$ denote the set of simple controls which are piecewise constant with grid $\pi$. We deduce from Lemma \ref{Vc} that, for all $\beta\in\BR(t+h)$,
\begin{multline}\label{egal-inf-continu} 
\inf_{u \in \UR^s(t+h)} J(t+h,p,q,u,\beta(u))=\inf_{\pi \in \Pi_{t+h}}\inf_{u \in \UR^\pi(t+h)} J(t+h,p,q,u,\beta(u)) \\=\inf_{\pi \in \Pi_{t+h}} \inf_{u \in \UR^c(t+h)\cap \UR^\pi(t+h)}J(t+h,p,q,u,\beta(u)).
\end{multline}
Let us fix $\pi=\{t+h=t_0<...<t_m=T\} \in \Pi_{t+h}$.
For all $j=0,...,m-1$, let $K^\pi_j$ be a compact subset of  $\Omega_{t+h,t_j}$ such that $\P_{t+h,t_j}(K^\pi_j) \geq 1- \tfrac{\varepsilon}{6CmT}$ so that $\P_{t+h}(\forall j=0,...,m-1, \, \omega|_{[t+h,t_j]} \in K^\pi_j) \geq 1- \tfrac{\varepsilon}{6CT}$.    
 Given $u \in \UR^c(t+h)\cap \UR^{\pi}(t+h)$, if we define $\tilde{u} \in \UR^\pi(t+h)$ by replacing the continuous maps $g_j$ by $f_j= g_j \ind_{K^\pi_j}$ for $j=0,...,m-1$, then for all $(p,q,\beta)$, we have
\[ | J(t+h,p,q,u,\beta(u)) - J(t+h,p,q,\tilde{u},\beta(\tilde{u})) | \leq \frac{\varepsilon}{3}.\]
As a consequence, for all $(p,q,\beta)$, we have
\begin{equation}
 \inf_{u \in \UR^c(t+h)\cap \UR^{\pi}(t+h)}J(t+h,p,q,\tilde{u},\beta(\tilde{u})) \leq \inf_{u \in\UR^c(t+h)\cap \UR^{\pi}(t+h)}J(t+h,p,q,u,\beta(u)) +\frac{\varepsilon}{3}.
\end{equation}
We can identify the set of processes $\tilde{u}$ when $u$ ranges through $\UR^c(t+h)\cap \UR^{\pi}(t+h)$ with $\prod_{j=0}^{m-1} C(K^\pi_j,U)$ through the map 
\[ \left(\prod_{j=0}^{m-1} C(K^\pi_j,U) \right) \times \Omega_{t+h} \ni (f_j)_{j=0,...,m-1}\mapsto \tilde{u} \in \UR^{\pi}(t+h),\]
which is Borel-measurable with $\tilde{u}(\omega_2)(s)=\sum_{j=0}^{m-1}\ind_{[t_j,t_{j+1})}(s) f_j(\omega_2|_{[t+h,t_j]})\ind_{K^\pi_j}(\omega_2|_{[t+h,t_j]})$. Using this identification, the previous inequality becomes
\begin{equation} \label{ineq-tilde} \inf_{(f_j)_{j=0,...,m-1} \in \left(\prod_{j=0}^{m-1} C(K^\pi_j,U) \right))}J(t+h,p,q,\tilde{u},\beta(\tilde{u}))\leq \inf_{u \in\UR^c(t+h)\cap \UR^{\pi}(t+h)}J(t+h,p,q,u,\beta(u)) +\frac{\varepsilon}{3} .
\end{equation}
\pa
By construction, the map 
\[  \Omega_{t,t+h}\times \Omega_{t+h} \times U^s_{t+h}\ni (\omega_1,\omega_2,\u_2) \mapsto \beta_{\omega_1}(\omega_2,\u_2)=\beta(\omega_1\otimes \omega_2, u^1(\omega_1)\otimes \u_2)|_{[t+h,T]} \in V^s_{t+h},\]
is Borel measurable. By composition and using Lemma \ref{SDE-borel}, it follows that 
\[  (\omega_1,\omega_2,(f_j)_{j=0,...,m-1}) \mapsto  \int_{t+h}^T H(s,\Phi^X_{t+h,T}(x,\omega_2,\tilde{u}(\omega_2))(s),\Phi^Y_{t+h,T}(y,\omega_2,\beta_{\omega_1}(\tilde{u}(\omega_2)))(s)) ds
,\]
is Borel-measurable from $\Omega_{t,t+h}\times \Omega_{t+h} \times \prod_{j=0}^{m-1} C(K^\pi_j,U)$ to $\R$. Applying Fubini's theorem, and using that
\begin{multline*}
 J(t+h,x,y,\tilde{u},\beta_{\omega_1}(\tilde{u})) \\
 =\int_{\Omega_{t+h}} \left( \int_{t+h}^T H(s,\Phi^X_{t+h,T}(x,\omega_2,\tilde{u}(\omega_2))(s),\Phi^Y_{t+h,T}(y,\omega_2,\beta_{\omega_1}(\tilde{u}(\omega_2)))(s)) ds \right) d\P_{t+h}(\omega_2),
\end{multline*}
we deduce finally that the map
\[ \Omega_{t,t+h}\times \prod_{j=0}^{m-1} C(K^\pi_j,U) \ni  (\omega_1,(f_j)_{j=0,...,m-1}) \mapsto J(t+h,x,y,\tilde{u},\beta_{\omega_1}(\tilde{u})),\]
is Borel measurable.
As this map is defined on a product of Polish spaces, we may apply Von Neumann's selection theorem (see Proposition 7.50 p.184 in \cite{bertsekas}, having in mind that Borel maps are lower-semi analytic and that analytically measurable maps are universally measurable)) to deduce the existence of a universally measurable selection $\tilde{u}^{2,\pi}$
\[ \begin{matrix} 
\Omega_{t,t+h}& \rightarrow &\prod_{j=0}^{m-1} C(K^\pi_j,U) \\ \omega_1 & \mapsto & \tilde{u}^{2,\pi}(\omega_1)
\end{matrix} ,\]
which is an $\varepsilon/3$-optimal best reply to $\beta_{\omega_1}$ in the considered class, i.e.:
\begin{equation}\label{ineq-selection}
 J(t+h,x,y,\tilde{u}^{2,\pi}(\omega_1),\beta_{\omega_1}(\tilde{u}^{2,\pi}(\omega_1)) \leq \inf_{(f_j)_{j=0,...,m-1} \in  \prod_{j=0}^{m-1} C(K^\pi_j,U)} J(t+h,x,y,\tilde{u},\beta_{\omega_1}(\tilde{u}))+ \frac{\varepsilon}{3}.
\end{equation}
This map is $\P_{t,t+h}$- almost surely equal to a Borel map (see Lemma 7.27 p.173 in \cite{bertsekas}) which, with a slight abuse of notation, will be also denoted $\tilde u^{2,\pi}(\omega_1)$.
\pa
Consider now the countable family of Borel maps $F_{\pi}$ indexed by $\pi \in \Pi_{t+h}$ defined by:
\[ \Omega_{t,t+h}\ni\omega_1 \mapsto F_\pi(\omega_1):=J(t+h,x,y,\tilde{u}^{2,\pi}(\omega_1),\beta_{\omega_1}(\tilde{u}^{2,\pi}(\omega_1))).\]
There exists a countable measurable partition $(D_\pi)_{\pi \in \Pi_{t+h}}$ of $\Omega_{t,t+h}$ such that the map $F=\sum_\pi F_\pi \ind_{D_\pi}$ verifies
\[ F \leq \inf_{\pi \in \Pi_{t+h}} F_\pi + \frac{\varepsilon}{3}, \]
and note that \eqref{egal-inf-continu}, \eqref{ineq-tilde} and \eqref{ineq-selection} imply
\begin{equation} \label{best-reply} F(\omega_1) \leq \inf_{\pi \in \Pi_{t+h}} F_\pi(\omega_1)  + \frac{\varepsilon}{3}\leq \inf_{u \in \UR^s_{t+h}} J(t+h,x,y,u,\beta_{\omega_1}(u))+\varepsilon.
\end{equation}
Let $(D_{\pi_1},...,D_{\pi_n},R)$ be a finite partition of $\Omega_{t,t+h}$ such that $\P_{t,t+h}(R) \leq \varepsilon$ and let us define $u$ by:
\[
 u_s(\omega)=u_s(\omega_1\otimes\omega_2)=\left\{
\begin{array}{ll}
u^1_s(\omega_1),&\mbox{ if }s\in[t,t+h),\\
u^2_s(\omega_1,\omega_2),& \mbox{ if }s\in[t+h,T],
\end{array}\right.\]
with $u^2(\omega_1,\omega_2)=\sum_{i=1}^n \tilde{u}^{2,\pi_i}(\omega_1)(\omega_2) \ind_{D_{\pi_i}}(\omega_1)$.
\pa
In order to show that the control $u$ is admissible, let us consider a time grid $\pi=\{ t=t_0<\ldots<t_m=T\}$ which coincides on $[t,t+h)$ with the time grid associated to $u^1$ and contains all the points of the partitions $\pi_1,\ldots,\pi_n$. Let us denote $\pi_i=\{t+h=t^i_0<...<t^i_{m_i}=T\}$ for $i=1,...,n$, and given any interval $[t_j,t_{j+1})$ of $\pi$, let $k(i,j)$ be the unique integer such that $ [t_j,t_{j+1}) \subset [t^i_{k(i,j)},t^i_{k(i,j)+1})$. Then $u$ admits the decomposition
\[ u(\omega,s)=\sum_{j=0}^{m-1}\ind_{[t_j,t_{j+1})}(s)g_j(\omega|_{[t,t_j]}),\]
which coincides with the decomposition of $u^1$ on $[t,t+h)$, and, for $s\in[t_j,t_{j+1})$ with $t+h\leq t_j$, the coefficient $g_j$ can be detailed  as follows:
\[ g_j(\omega|_{[t,t_j]})=\sum_{i=1}^n \ind_{D_{\pi_i}}(\omega_1)g_{i,k(i,j)}(\omega_1)(\omega_2|_{[t+h,t^i_{k(i,j)}]}),\]
where the term $g_{i,k(i,j)}$ appears in the decomposition of $\tilde{u}^{2,\pi_i}(\omega_1)$:
\[ \tilde{u}^{2,\pi_i}(\omega_1)(\omega_2,s)=\sum_{k=0}^{m_i-1}\ind_{[t^i_k,t^i_{k+1})}(s)g_{i,k}(\omega_1)(\omega_2|_{[t+h,t^i_k]}).\]
To conclude, note that $u$ coincides with $u^1$ on $[t,t+h)$ and that by construction for all $\omega_1 \notin R$, thanks to \eqref{best-reply}, $u^2(\omega_1,.)$ is an $\varepsilon$-best reply to the continuation strategy $\beta_{\omega_1}$ in the game starting at time $t+h$ with initial conditions $(x,y)$.
\end{proof}


\noindent {\bf Acknowledgements:} 
We like to thank the referees for their helpful comments and suggestions.

\end{document}